\documentclass[numbers,webpdf]{ima-authoring-template}%
\usepackage{mathtools}
\usepackage{caption}
\allowdisplaybreaks

\graphicspath{{Fig/}}

\theoremstyle{thmstyletwo}
\newtheorem{theorem}{Theorem}
\newtheorem{proposition}[theorem]{Proposition}

\newtheorem{corollary}[theorem]{Corollary}
\newtheorem{assumption}[theorem]{Assumption}

\newtheorem{remark}{Remark}

\numberwithin{equation}{section}

\DeclareUnicodeCharacter{0308}{\"o}
\usepackage{hyperref}
\usepackage{nameref}
\usepackage{algpseudocode}
\usepackage{braket,amsfonts}
\usepackage{amsmath,amssymb,amsxtra}
\usepackage{bm,mathrsfs,marvosym,eurosym}
\usepackage{mathtools,enumerate}
\usepackage{caption}
\usepackage{booktabs}
\usepackage{threeparttable}
\usepackage{comment}
\usepackage{siunitx}
\usepackage{nccmath}
\usepackage{bbm}
\usepackage{array}
\usepackage{multirow}
\usepackage[normalem]{ulem}
\usepackage{upgreek}
\usepackage{tikz}
\usepackage{url}       
\usetikzlibrary{arrows,calc,decorations.pathreplacing}
\usepackage[caption=false]{subfig}

\usetikzlibrary{arrows.meta, positioning}
\usepackage{algorithmicx}

\usepackage{graphicx}
\graphicspath{{./pics/}}

\def\d{{\mathrm d}}

\def\G{{\mathcal G}}
\def\F{{\mathcal F}}

\def\0{{\textbf{0}}}
\def\L{{\mathcal{L}}}

\def \II{(\Omega)}

\def \L{\mathcal{L}}
\def \t{\theta}
\def\bt{\boldsymbol{\t}}

\def\<{{\langle }}
\def\>{{\rangle }}

\allowdisplaybreaks

\begin{document}

\DOI{DOI HERE}
\copyrightyear{2026}
\vol{00}
\pubyear{2024}
\access{Advance Access Publication Date: Day Month Year}
\appnotes{Paper}
\copyrightstatement{Published by Oxford University Press on behalf of the Institute of Mathematics and its Applications. All rights reserved.}
\firstpage{1}

\title[Optimized Two-Step Coarse Propagators in Parareal Algorithms]{Optimized two-step coarse propagators in parareal Algorithms\thanks{The work of G. Li is supported by Hong Kong Research Grants Council (Project 17317022). The work of K. Zhang is supported by the NSF of China under the grant No. 12271207, and by the fundamental research funds for the central universities. The work of Z. Zhou is supported by National Natural Science Foundation of China (Project 12422117), Hong Kong Research Grants Council (15302323) and an internal grant of Hong Kong Polytechnic University (Project ID: P0053938, Work Programme: 4-ZZVA).}}

\author{Guanglian Li
\address{\orgdiv{Department of Mathematics}, \orgname{The University of Hong Kong}, \orgaddress{Hong Kong SAR, \country{P. R. China} (\href{mailto:lotusli@maths.hku.hk}{lotusli@maths.hku.hk})}}}

\author{Qingle Lin 
\address{\orgdiv{Department of Applied Mathematics}, \orgname{The Hong Kong Polytechnic University}, \orgaddress{ Hong Kong SAR, \country{P. R. China} (\href{mailto:qingle.lin@connect.polyu.hk}{qingle.lin@connect.polyu.hk}) }}}

\author{Kai Zhang
\address{\orgdiv{School of Mathematics}, \orgname{Jilin University}, \orgaddress{Changchun 130012, \country{P. R. China} (\href{mailto:zhangkai7609@gmail.com}{zhangkai7609@gmail.com})}}}

\author{Zhi Zhou
\address{\orgdiv{Department of Applied Mathematics}, \orgname{The Hong Kong Polytechnic University}, \orgaddress{  Hong Kong SAR, \country{P. R. China} (\href{mailto:zhizhou@polyu.edu.hk}{zhizhou@polyu.edu.hk})}}}

\authormark{G. Li, Q. Lin, K. Zhang, and Z. Zhou}

\received{Date}{0}{Year}
\revised{Date}{0}{Year}
\accepted{Date}{0}{Year}


\abstract{
{In this work, we propose a novel framework for accelerating the parareal algorithm, in which the coarse propagator is formulated as a two-step method and optimized  with respect to the convergence factor.} We derive a rigorous error estimate for the proposed two-step parareal algorithm, yielding an explicit bound on the linear convergence factor. This estimate is not only of theoretical interest: it provides a quantitative guideline for selecting and designing coarse propagators. Guided by this estimate, we {consider the linear parabolic equation as an illustrative example and }construct an optimized two-step coarse propagator~(O2CP) that delivers very fast convergence in practice. The resulting method attains an optimized convergence factor of approximately $0.0064$, substantially smaller than that of commonly used practical coarse propagators in the classical parareal setting, while keeping the computational cost moderate. Numerical experiments on linear and nonlinear parabolic equations fully support the theoretical analysis and demonstrate rapid convergence of the two-step parareal algorithm equipped with the O2CP.
}


\maketitle

\ 

\section{Introduction}

The numerical approximation of  time evolution problems  often employs time stepping schemes, which advance step by step in time. The sequential nature constitutes a major computational bottleneck, particularly for long-time simulations. Moreover, the hardware development has reached physical limits with respect to clock speed, and further increases in computational performance are possible only through the use of more cores~\cite{gander2025time}.

To overcome the bottleneck, various parallel-in-time (PinT) algorithms have been developed~\cite{Gander2024timeparallel}. Notable examples include parareal method \cite{LionsMadayTurinici:2001}, ParaDiag method~\cite{wu2018toward,GanderEtAl2021,gander2025time,wu2021parallel,MR2385067}, multigrid-reduction-in-time~\cite{dobrev2017two,falgout2014parallel,MR3716560}, and parallel full approximation scheme in space and time~\cite{MR2979518,MR3504550} etc. The parareal algorithm due to Lions, Maday, and Turinici~\cite{LionsMadayTurinici:2001} combines a computationally inexpensive coarse propagator (CP) with an accurate yet expensive fine propagator (FP). The CP generates a low-resolution approximation of the solution over large time steps, while the FP provides high-fidelity solutions on finer temporal subintervals. The CP is first applied sequentially to produce an initial guess, after which the FP is executed in parallel on each subinterval to iteratively refine the solution. Ideally, the algorithm achieves the accuracy of the FP solution at a cost comparable to that of the coarse propagators alone.

In practice, the FPs are typically chosen to be L-stable. The specific choice of the FP has little influence on the convergence behavior and is often assumed to be exact in most theoretical analyses. The choice of the CP greatly impacts the performance of the parareal method. Existing studies have mostly focused on single-step schemes as the CP. There are two primary criteria for evaluating the CP: the convergence factor $\gamma_{\ell}$ and the cost. The backward Euler (BE) method is inexpensive, but its convergence factor $\gamma_{\ell} \approx 0.298$ is relatively large. Several high-order L-stable schemes exhibit small convergence factor (e.g., three-stage Lobatto IIIC with $\gamma_{\ell} \approx 0.024$), but they are typically expensive to implement. One popular option is the two-stage Lobatto IIIC or SDIRK methods as the CP, which enjoys a relatively small convergence factor at moderate cost. Multi-step time-stepping methods offer excellent stability, high resolution, and low cost per step, but they have scarcely been explored. Audouze et al ~\cite{audouze} used single-step CPs to update more values within subintervals and Ait-Ameur and Maday \cite{ait2023multi,MR4167579} studied the parareal method with BDF$k$ schemes as FPs. The iteration structure for multi-step CPs has not been developed so far. These observations motivate the following question: \textit{Is it possible to achieve substantially faster convergence of the parareal algorithm using multi-step CPs?}

In this work, we consider the linear parabolic equation as an illustrative example. We address the above question by proposing a novel two-step parareal algorithm together with an easy-to-implement strategy for selecting the CPs.  The two-step formulation can be viewed as an extension of the primary block-iteration framework~\cite{gander2023unified}. We derive a rigorous error estimate and an explicit linear convergence factor for the two-step iteration. Our analysis also shows that simply adopting a standard linear multi-step scheme (e.g., BDF2) as the CP does not, by itself, guarantee fast convergence. This motivates the construction of an optimized two-step coarse propagator (O2CP), obtained by minimizing the upper bound on the convergence factor derived in our convergence analysis, subject to a consistency constraint. The resulting O2CP is A-stable and achieves an optimized convergence factor of approximately $0.0064$, markedly smaller than that of commonly used practical CPs in the classical parareal framework, while requiring only two Poisson solves per coarse time step. Numerical experiments for linear and semilinear parabolic problems corroborate the theory: the linear tests match the predicted convergence factor, and the nonlinear tests exhibit rapid convergence even under strong nonlinearity.

Recently, optimizing and learning-based techniques have become increasingly common for the parareal method, including stage-parallel preconditioners for Radau IIA methods~\cite{MunchEtAl2023}, 
spectral deferred correction preconditioners~\cite{CaklovicEtAl2025}, stochastic parareal variants that select the smoothest trajectory~\cite{pentland2022stochastic}, coarse-grid operator optimization for MGRIT~\cite{DeSterck:2021,yoda2024coarse}, physics-informed neural networks as CPs~\cite{IbrahimRuprecht:2023}, random neural networks to represent the discrepancy between coarse and fine solutions~\cite{gattiglio2024randnet}, and optimized coarse propagators~(OCPs)~\cite{jin2025optimizing}. The construction of the O2CP builds on the rigorous error estimator and preserves the stability and consistency. 

The rest of the paper is structured as follows. In Section~\ref{sec:prelim}, we review single-step and multi-step solvers and the classical parareal algorithm. In Section~\ref{sec:two-step}, we propose two-step parareal algorithms based on block iterations and derive error estimates. In Section~\ref{sec:optim}, we construct the O2CP to accelerate parareal iterations. Finally, in Section~\ref{sec:num}, we show the performance of the O2CP on linear and semilinear parabolic problems. Throughout, the notation $(\cdot,\cdot)$ denotes the $L^2(\Omega)$ inner product. Let $\{(\lambda_p,\varphi_p)\}_{p=1}^\infty$ be the eigenpairs of $A$, where $\{\varphi_p\}_{p=1}^\infty$ forms an orthonormal basis of $L^2(\Omega)$. We denote by $\mathcal{L}$ the space of bounded linear operators on $L^2(\Omega)$, equipped with the operator norm $\|\cdot\|_{\mathcal{L}}$.

\section{Preliminaries on parareal algorithm}\label{sec:prelim}
Throughout the theoretical part of this paper, we consider the linear parabolic equation as an illustrative example. In this section, we introduce the single-step and multi-step solvers and present the standard parareal algorithm.

In particular, we let $T >0$ be a fixed final time, and $\Omega$ be   a Lipschitz domain in $\mathbb{R}^d$ with $d\ge 1$. Let
$A$ be a positive definite, self-adjoint linear elliptic operator on the Hilbert space $(L^2\II,(\cdot, \cdot))$ with its domain $D(A)$ dense in $L^2 \II$. Fix $f\in L^2(0, T; L^2 \II)$ and $u_0 \in L^2\II$. Consider the following initial value problem: Find $u \in C((0,T];D(A))\cap C([0,T];L^2 \II)$ such that
\begin{equation}\label{eqn:pde}
\left\{\begin{aligned}
u'(t)+Au(t)&=f(t), \quad 0<t<T, \\
u(0)&=u_0.
\end{aligned}\right.
\end{equation}

A single-step scheme approximates the solution $U_{n+1}$ with the initial value $U_n$ and time step $\Delta T$ by
\begin{equation}\label{eqn:semi}
\overline{U}_{n+1} = R({\Delta T} A) U_{n} + {\Delta T} \sum_{i=1}^m P_i({\Delta T} A)
f(T_n+c_i\Delta T) ,\quad \text{for}~0\le n\le N_c-1,
\end{equation}
where $R(s)$ and the sequence $\{P_i(s)\}^m_{i=1}$ are rational functions, and ${c_i}$ are distinct real numbers in $[0,1]$.
A linear $k$-step scheme approximates the solution $U_{n+k}$ with $k$ initial values $U_{n},\cdots, U_{n+k-1}$ by $\overline{U}_{n+k}$:
\begin{equation}\label{eqn:k-step}
(\alpha_{k} +\beta_{k} \Delta TA)\overline{U}_{n+k} +\sum_{i=0}^{k-1} \left( \alpha_{i} +\beta_{i} \Delta TA \right) U_{n+i}=\Delta T\sum_{i=0}^{k} \beta_{i} f\left( T_{n}+i\Delta T \right) .
\end{equation}
Throughout we assume that the scheme \eqref{eqn:k-step} is of order $q$ in the sense that 
\begin{equation}\label{eqn:order_k_step}
    \sum_{i=0}^{k} \alpha_i = 0\quad{\text{and}}\quad \sum_{i=0}^k \alpha_i i^p = p \sum_{i=0}^k \beta_i i^{p-1}, \quad \text{for} ~p=1,\cdots,q. 
\end{equation}
For details on multi-step methods, see \cite[Chapter III]{hairer1993solving}. The BDF2 scheme, with the coefficients $\alpha_0 = 1/2$, $\alpha_1 = -2$, $\alpha_2=3/2$, $\beta_0=\beta_1 = 0$ and $\beta_{2}=1$,  satisfies the order condition \eqref{eqn:order_k_step} with $q=2$.  

Next, we describe the parareal algorithm for problem~\eqref{eqn:pde}, based on the single-step solver \eqref{eqn:semi}. We divide the time interval $(0,T)$ into $N$ equidistant
subintervals, each of length ${\Delta t} = T/N$. Let $\Delta T = J \Delta t~(J\in 2\mathbb{N}_{+})$ be the coarse step size, $N_c = T/\Delta T\in \mathbb{N}$, and $T_n = n\Delta T$.

The numerical solvers $\G$ and $\F$ operate on the coarse and fine time grids, respectively. Typically, $\G$ is a computationally cheap and low-order method, whereas $\F$ is a high-order but expensive single-step scheme \eqref{eqn:semi}.  
 Given the initial data $v\in L^2 \II$, the CP $\G$ and the FP $\F$ are respectively defined by
\begin{align}
	\G(T_n,\Delta T,v) &= R({\Delta T} A) v + {\Delta T} \sum_{i=1}^{M} P_i({\Delta T} A) f(T_n+C_i\Delta T) , \label{eqn:coarse-integrator}\\
	\F(t_n,\Delta t,v) &= r({\Delta t} A) v + {\Delta t} \sum_{i=1}^{m} p_i({\Delta t} A) f(t_n+c_i\Delta t) ,\label{eqn:fine-integrator}
\end{align}
where $M$ and $m$ denote the numbers of stages of the CP $\mathcal{G}$ and the FP $\mathcal{F}$, respectively, and $R$ and $P_i$ are rational functions for the CP $\mathcal{G}$, and $r$ and $p_i$ are for the FP $\mathcal{F}$. The parareal iteration is given in Algorithm \ref{alg:para}.

\begin{algorithm}[ht]
\centering
\caption{Single-step parareal algorithm}\label{alg:para}
\begin{algorithmic}[1]
\State \textbf{Initialization}: Compute $U^0_{n+1} = \mathcal{G}(T_n,\Delta T,U_n^0)$ with $U_0^0=u_0$, $n=0,\dotsc,N_c-1$.
\For{$k=0,1,\dotsc,K$}
    \State On each subinterval $[T_n,T_{n+1}]$, sequentially compute for $j=0,1,\dotsc,J-1$,
    \State \begin{equation}\label{eqn:FP_parallel_}
        \widetilde U_{n,j+1}^{k}
        = \mathcal{F}(T_n+j\Delta t,\Delta t,\widetilde U_{n,j}^{k}).
    \end{equation}
    with initial value $\widetilde U_{n,0}^{k}=U_n^k$, and set
    $\widetilde U_{n+1}^{k}=\widetilde U_{n,J}^{k}$.
    
    \State Perform sequential corrections, by computing $U^{k+1}_{n+1}$ by
    \State \begin{equation*}
        U^{k+1}_{n+1}
        = \mathcal{G}(T_n,\Delta T,U^{k+1}_{n})
        + \widetilde U_{n+1}^{k}
        - \mathcal{G}(T_n,\Delta T,U^{k}_{n}).
    \end{equation*}
    with $U^{k+1}_0=u_0$, for $n=0,1,\dotsc,N_c-1$.
    
    \State Check the stopping criterion.
\EndFor
\end{algorithmic}
\end{algorithm}

The convergence of the parareal algorithm for linear parabolic PDEs has been extensively studied~\cite{gander2007analysis,WuZhou:2015,Wu:IMA2015,dobrev2017two,hessenthaler2020multilevel,legoll2013micro,jin2025optimizing}. Below we discuss the case when the FP is an exact solver $\Phi_E$: for $v\in L^2\II$, $\Phi_E (v)$ is given by $\Phi_E (T_n,\Delta T,v)= w(\Delta T)$, with $w(\Delta T)$ solving problem \eqref{eqn:pde} with the initial data $v$ at $T_n$
\begin{equation}\label{eqn:exp}
\Phi_E(T_n,\Delta T,v) := w(\Delta T) = e^{- A \Delta T} v + \int_{0}^{\Delta T} e^{ -A(\Delta T-s) } f( T_n +s) \, \d s.
\end{equation}
Then we define the exact solution $U_n$ recursively via $U_{n+1} = \Phi_E (T_n,\Delta T,U_n)$ for $n=0,\dotsc,N_c-1,$ with $U_0=u_0$.

\section{Two-step parareal algorithm}\label{sec:two-step}
Motivated by the efficiency of multi-step methods \eqref{eqn:k-step} for stiff problems, which attain high accuracy with fewer function evaluations than single-step methods, we now extend the parareal method to a two-step structure, and derive error estimates, which form the basis of constructing OCPs.
\subsection{Algorithm description}
Since a two-step solver requires two initial values, we assume that $J\in 2\mathbb{N}$ and include an update at $T_{n+1/2}$ within each coarse time interval $[T_n,T_{n+1}]$. Gander et al \cite{gander2023unified} proposed the primary block iteration and $kn$-graph for Dahlquist's test equation $u'=\lambda u$. Fig.~\ref{fig: kn-graph}(a) shows the $kn$-graph for the single-step parareal algorithm:
\begin{equation*}
    U_{n+1}^{k+1} = R(\Delta TA)U_n^{k+1} + (r(\Delta t A)^J - R(\Delta TA)) U_n^k.\end{equation*}
The horizontal and vertical axes denote the time-step index $n$ and the parareal iteration index $k$, respectively. The node at position $(n,k)$ denotes  $U_n^k$. The diagram shows that the corrections $(\F - \G)(U_n^k)~(n = 1, \dots, N_c)$ can be evaluated in parallel.

\begin{figure}[ht]
\centering
\begin{minipage}{0.46\textwidth}
\centering
\begin{tikzpicture}[scale=1.4, dot/.style={circle,draw,fill=white,minimum size=2mm,inner sep=0pt},>={Stealth[length=2mm]}]
\draw[thick,-{Stealth[length=3mm]}] (0,0) -- (0,2) ;
\draw[thick,-{Stealth[length=3mm]}] (0,0) -- (2,0) ;
\node[dot] at (0.3,0.3) (00) {};
\node[dot] at (0.3,1.7) (01) {};
\node[dot] at (1.7,0.3) (10) {};
\node[dot] at (1.7,1.7) (11) {};
\draw[thick,-{Stealth[length=2mm]}] (0.4,0.4) -- (1.6,1.6) ;
\draw[thick,-{Stealth[length=2mm]}] (0.4,1.7) -- (1.6,1.7) node[midway,above] {$\G$} ;
\draw (0.6,1.1) node {$\F - \G$} ;
\draw (0.3,-0.1) node[below] {$n$};
\draw (-0.1,0.3) node[left] {$k$};
\draw (1.7,-0.1) node[below] {$n+1$};
\draw (-0.1,1.7) node[left] {$k+1$};
\end{tikzpicture}
\par\vspace{10pt}
\textbf{(a)}
\end{minipage}%
\hfill
\begin{minipage}{0.52\textwidth}
\centering
\begin{tikzpicture}[scale=1.3, dot/.style={circle,draw,fill=white,minimum size=2mm,inner sep=0pt},>={Stealth[length=2mm]}]
\draw[thick,-{Stealth[length=3mm]}] (0,0) -- (0,2) ;
\draw[thick,-{Stealth[length=3mm]}] (0,0) -- (3.5,0) ;
\node[dot] at (0.3,0.3) (00) {};
\node[dot] at (0.3,1.7) (01) {};
\node[dot] at (1.7,0.3) (10) {};
\node[dot] at (1.7,1.7) (11) {};
\node[dot] at (3.1,0.3) (20) {};
\node[dot] at (3.1,1.7) (21) {};
\draw[->] (0.4,1.8) to[bend left=45] (3,1.8);
\draw[thick,-{Stealth[length=2mm]}] (0.4,0.4) -- (3,1.63) node[midway,below right] {$\F^2-\G_{2}\F - \G_1$} ;
\draw[thick,-{Stealth[length=2mm]}] (1.8,1.7) -- (3,1.7) node[midway,above] {$\G_2$} ;
\draw (1.7,2.1) node {$\G_1$} ;
\draw (0.3,-0.1) node[below] {$\frac{m}{2}$};
\draw (-0.1,0.3) node[left] {$k$};
\draw (1.7,-0.1) node[below] {$\frac{m}2+\frac12$};
\draw (-0.1,1.7) node[left] {$k+1$};
\draw (3.1,-0.1) node[below] {$\frac{m}{2}+1$};
\end{tikzpicture}
\par\vspace{10pt}
\textbf{(b)}
\end{minipage}
\caption{$kn$-graphs for the parareal method (a) and two-step parareal method (b).}
\label{fig: kn-graph}
\end{figure}
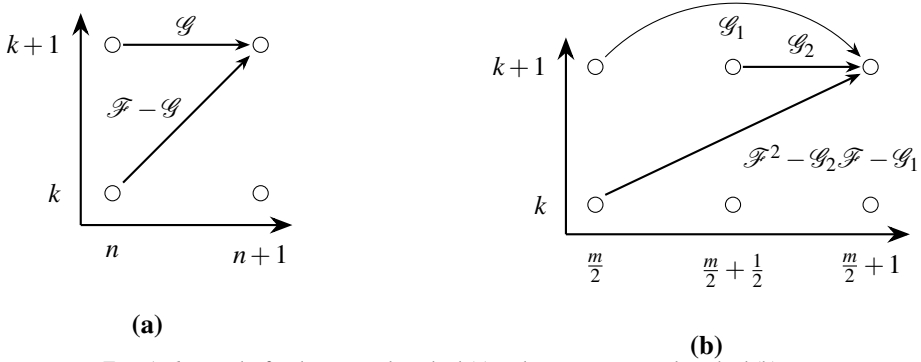
Consider the following two-step structure for $m=0,1,\dotsc,2N_c-2$,
\begin{equation}\label{eqn:multi-step structure}
U_{{m}/{2}+1}^{k+1} = R_1 (\Delta TA) U_{{m}/{2}}^{k+1} + R_2 (\Delta TA)U_{{(m+1)}/{2}}^{k+1} + B_{-1}^0(\Delta TA) U_{{m}/{2}}^k,
\end{equation}
where $R_1$ and $R_2$ are suitable rational functions, and $B_{-1}^{0}: L^2 \II \rightarrow L^2 \II$ is linear. To ensure the consistency of \eqref{eqn:multi-step structure} with the single-step fine solution, we enforce 
\begin{equation*}
    B_{-1}^0(\Delta TA) = r(\Delta tA)^{J} - R_2(\Delta TA) r(\Delta tA)^{{J}/{2}} - R_1 (\Delta TA).
\end{equation*}
It ensures the finite iteration convergence property of the two-step parareal algorithm in Theorem~\ref{thm:finite conv.} below.
Fig.~\ref{fig: kn-graph}(b) shows the $kn$-graph of the two-step parareal algorithm with $\G_1 = R_1$ and $\G_2 = R_2$. To extend the algorithm to inhomogeneous problems, we define a two-step CP $\G^*$ for two initial values $v_1,v_2 \in L^2 \II$ (at $T_n-\tau$ and $T_n$) by
\begin{equation}\label{eqn:multi-step CP}
  \G^*(T_n,\tau,v_1,v_2) = R_1(\tau A) v_1 + R_2 (\tau A)v_2 + \tau (\alpha_2 + \beta_2 \tau A)^{-1} \sum_{i=0}^2 \beta_i f(T_n + (i-1)\tau ), 
\end{equation}
with the rational functions $R_1$ and $R_2$ given by 
\begin{equation}\label{eqn:R_a_b}
    R_1(s) = \frac{-\alpha_0 -\beta_0 s}{\alpha_2 + \beta_2 s} \quad \mbox{and} \quad R_2 (s)=\frac{-\alpha_1 - \beta_1 s}{\alpha_2 + \beta_2 s}.
\end{equation}
Then we present the two-step parareal algorithm in Algorithm~\ref{alg:two-step parareal}.
\begin{algorithm}[htbp!]

\centering
\caption{Two-step parareal algorithm}\label{alg:two-step parareal}
\begin{algorithmic}[1]
\State \textbf{Initialization}: Compute $U^0_{1/2} = \mathcal{G}(T_0,\Delta T /2,U_0^0)$ with $U_0^0=u_0$, and
\[
U^0_{m/2+1}
= \mathcal{G}^*(T_{m/2},\Delta T/2,U^0_{m/2},U^0_{(m+1)/2}),
\qquad m=0,1,\dotsc,2N_c-2.
\]

\For{$k=0,1,\dotsc,K$}
    \State On each subinterval $[T_{m/2},T_{m/2+1}]$, sequentially compute for $j=0,1,\dotsc,J-1$,
    \begin{equation}\label{eqn:FP_parallel}
    \widetilde U_{m/2,j+1}^{k}
    = \mathcal{F}(T_{m/2}+j\Delta t,\Delta t,\widetilde U_{m/2,j}^{k}),
    \end{equation}
    with initial value $\widetilde U_{m/2,0}^{k}=U^k_{m/2}$, and set
    $\widetilde U_{m/2+1}^{k}=\widetilde U_{m/2,J}^{k}$.

    \If{$k=0$}
        \State Set $U_{1/2}^0=\widetilde U_{0,J/2}^0$ and $U_1^0=\widetilde U_1^0$.
    \EndIf

    \State Perform sequential corrections, i.e., compute $U^{k+1}_{m/2+1}$ by
    \begin{equation}\label{eqn:two-step parareal}
    \begin{aligned}
    U^{k+1}_{m/2+1}
    &= \mathcal{G}^*(T_{m/2},\Delta T/2,U^{k+1}_{m/2},U^{k+1}_{(m+1)/2})
      + \widetilde U_{m/2+1}^{k} \\
    &\quad
      - \mathcal{G}^*(T_{m/2},\Delta T/2,U^{k}_{m/2},\widetilde U^{k}_{m/2,J/2}),
    \end{aligned}
    \end{equation}
    with $U^{k+1}_0=u_0$ and $U_{1/2}^{k+1}=U_{1/2}^0$, for $m=1,\dotsc,2N_c-2$.

    \State Check the stopping criterion.
\EndFor
\end{algorithmic}
\end{algorithm}
\begin{remark}
The FP values $\widetilde{U}_{m/2}^{k}$~\eqref{eqn:FP_parallel} are computed in parallel within each coarse interval $[T_{{m}/{2}},T_{{m}/{2}+1}]$, for $m=0,1,\dotsc,2N_c-2,$ and the costs of~\eqref{eqn:FP_parallel} and~\eqref{eqn:FP_parallel_} are identical. However, due to the overlapping of the coarse intervals, $2N_c$ processors are required to implement Algorithm~\ref{alg:two-step parareal}. The two-step CP is implemented sequentially $2N_c$ times at each iteration. Nevertheless, its cost is on par with the BE, and is much lower than high-order L-stable single-step methods. 
\end{remark}

\begin{remark}
If $R_1=0$ and $R_2$ is the stability function of a single-step solver~\eqref{eqn:semi} in \eqref{eqn:multi-step CP}, then Algorithm~\ref{alg:two-step parareal} reduces to the parareal algorithm with FCF-relaxation~\cite{dobrev2017two} and half the coarse step size. 

\end{remark}

We define the fine solution $U_n$ recursively by 
\begin{equation*}
U_{n,j+1}=\F\left( T_{n}+j\Delta t,\Delta t,U_{n,j} \right),\quad\text{for}~j=0,1,\dotsc,J-1,
\end{equation*}
with ${U}_{n,0}=U_{n}$ and initial condition $U_0=u_0$. Setting $U_{n+1}:=U_{n,J}$, we also denote the midpoint value by $U_{n+1/2}:=U_{n,J/2}$. Finally, we introduce the compact notations 
\begin{align*}
\F_{{\Delta T}/{2}}(T_n,U_n):=U_{n+1/2}\quad\text{and}\quad \F_{\Delta T}(T_n,U_n):=U_{n+1},
\end{align*}
which summarize the integration over half a coarse step and a full coarse step, respectively.
The next result gives a finite-step convergence of Algorithm~\ref{alg:two-step parareal}.
\begin{theorem}\label{thm:finite conv.}
For Algorithm~\ref{alg:two-step parareal}, $U_{
{m}/{2}}^k=U_{{m}/{2}}$ holds when $m\leq 2k$.
\end{theorem}
\begin{proof}
We prove the assertion by mathematical induction on $k$. The cases $k=0,1$ hold trivially. Suppose that $U_{{m}/{2}}^k=U_{{m}/{2}}$ holds for $m \leq 2k$. Then for $k+1$, by letting $m=2k-1$ in~\eqref{eqn:two-step parareal}, we obtain 
\begin{align*}
U^{k+1}_{k+1/{2}} &= \G^*(T_{k-{1}/{2}}, {\Delta T}/{2}, U^{k+1}_{k-{1}/{2}}, U^{k+1}_{k})  + \widetilde{U}_{k+{1}/{2}}^{k}\\
&\quad -  \G^*(T_{k-{1}/{2}},{\Delta T}/{2}, U^{k}_{k-{1}/{2}}, \F_{{\Delta T}/{2}}(T_{k-{1}/{2}},U_{k-{1}/{2}}^k))\\
&= \G^*(T_{k-{1}/{2}}, {\Delta T}/{2}, U_{k-{1}/{2}}, U_{k})  + \F_{\Delta T}(T_{k - {1}/{2}},U_{k - {1}/{2}}) \\
&\quad -  \G^*(T_{k -{1}/{2}}, {\Delta T}/{2}, U_{k - {1}/{2}}, U_{k})\\
&=  \F_{\Delta T}(T_{{k - {1}/{2}}},U_{k - {1}/{2}}) = U_{k + {1}/{2}}.
\end{align*}
The argument also holds for $U_{k+1}^{k+1}$ since $\widetilde{U}_{k+1}^k = \F_{\Delta T} (T_k,U_k)=U_{k+1}.$
\end{proof}

\subsection{Error estimation} 
We present an error bound on the two-step parareal algorithm in Algorithm~\ref{alg:two-step parareal}. 
\begin{theorem}\label{thm:error estimation}
Consider Algorithm \ref{alg:two-step parareal} with the CP~\eqref{eqn:multi-step CP} and the FP~\eqref{eqn:fine-integrator}. Let $\rho_1 (s)$ and $\rho_2 (s)$ be defined by \begin{equation}\label{eqn:rho}
        \rho_1(s)+\rho_2 (s)=R_2(s)\quad \mbox{and}\quad \rho_1(s)\rho_2 (s)=-R_1 (s),
\end{equation} 
and suppose $|\rho_1(s)|,|\rho_2(s)|<1$ for $s\in \mathbb{R}_+$. Define $\gamma^*$ the convergence factor by
\begin{equation*}
\gamma^*(r,R_1,R_2,J) = \sup_{s\in\mathbb{R}_+}\frac{|r({2s}/{J})^J-R_2(s)r({2s}/{J})^{{J}/{2}}-R_1(s)|}{(1-|\rho_1(s)|)(1-|\rho_2(s)|)}.
\end{equation*}
Then there exists $c$ independent of $k$ such that 
\begin{equation*}
\max_{0 \leq m\leq 2N_c} \|U_{{m}/{2}}^k-U_{{m}/{2}}\|_{L^2\II} \leq c\gamma^* (r,R_1,R_2,J)^k.
\end{equation*}

\end{theorem}
\begin{proof}
The two-step parareal algorithm reads
\begin{equation*}
\begin{aligned}
U^{k+1}_{{m}/{2}+1} &= \G^*(T_{{m}/{2}}, {\Delta T}/{2}, U^{k+1}_{{m}/{2}}, U^{k+1}_{{(m+1)}/{2}})  + \F_{\Delta T}(T_{{m}/{2}},U_{{m}/{2}}^k) \\
&\quad -  \G^*(T_{{m}/{2}}, {\Delta T}/{2}, U^{k}_{{m}/{2}}, \F_{{\Delta T}/{2}}(T_{{m}/{2}},U_{{m}/{2}}^k)),
\end{aligned}
\end{equation*}
and the fine solution $U_{{m}/{2}+1}$ satisfies
\begin{equation*}
\begin{aligned}
    U_{{m}/{2}+1} &= \G^*(T_{{m}/{2}}, {\Delta T}/{2}, U_{{m}/{2}}, U_{{(m+1)}/{2}})  + \F_{\Delta T}(T_{{m}/{2}},U_{{m}/{2}})\\
    &\quad -  \G^*(T_{{m}/{2}}, {\Delta T}/{2}, U_{{m}/{2}}, \F_{{\Delta T}/{2}}(T_{{m}/{2}},U_{{m}/{2}})),
    \end{aligned}
\end{equation*}
Let the error $E_{{m}/{2}}^k=U_{{m}/{2}}^k - U_{{m}/{2}}$ for $m=0,1,\dotsc,2N_c$. Taking the difference of the last two equations gives
\begin{align*}
 E_{{m}/{2}+1}^{k+1}&=R_1({\Delta TA}/{2})E_{{m}/{2}}^{k+1} + R_2({\Delta TA}/{2})E_{{(m+1)}/{2}}^{k+1} + r(\Delta t A)^J E_{{m}/{2}}^k \\
    &\quad -R_1({\Delta TA}/{2})E_{{m}/{2}}^{k} - R_2({\Delta TA}/{2})r(\Delta tA)^{{J}/{2}}E_{{m}/{2}}^{k}.
\end{align*}
Let $\delta T={\Delta T}/{2}$. Then the recursion for the error $E^k_{m/2}$ is given by
\begin{equation}\label{eqn:E_n^k}
\begin{aligned}
E_{{m}/{2}+1}^{k+1} &= R_1(\delta TA)E_{{m}/{2}}^{k+1} + R_2(\delta TA)E_{{(m+1)}/{2}}^{k+1} \\
&\quad + \big(r({2\delta TA}/{J})^J-r({2\delta TA}/{J})^{{J}/{2}}R_2(\delta TA)-R_1 (\delta TA)\big)E_{{m}/{2}}^k.
\end{aligned}
\end{equation}
Let $(\lambda_p,\varphi_p)$ be an eigenpair of $A$, and $e_{{m}/{2},p}^k:=(E_{{m}/2}^k,\varphi_p)$. Then, with $d_p = \delta T\lambda_p$, testing Equation \eqref{eqn:E_n^k} by $\varphi_p$ gives
\begin{equation*}
\begin{aligned}
e_{{m}/{2}+1,p}^{k+1} &= R_1(d_p)e_{{m}/{2},p}^{k+1} + R_2(d_p)e_{{(m+1)}/{2},p}^{k+1} \\
&\quad + \big(r({2d_p}/{J})^J-r({2d_p}/{J})^{{J}/{2}}R_2(d_p)-R_1 (d_p)\big) e_{{m}/{2},p}^k.
\end{aligned}
\end{equation*}
By the definitions of $\rho_1(s)$ and $\rho_2(s)$ in \eqref{eqn:rho}, we obtain the recursion
\begin{align}\label{eqn:iter1}
    e_{{m}/{2}+1,p}^{k+1} - \rho_2 (d_p) e_{(m+1)/2,p}^{k+1} =  \rho_1(d_p) \big(e_{(m+1)/2,p}^{k+1} - \rho_2 (d_p) e_{{m}/{2},p}^{k+1}\big) + I_{{m}/{2},p}^k,
\end{align}
with $I_{{m}/{2},p}^k := \big(r({2d_p}/{J})^J-r({2d_p}/{J})^{{J}/{2}}R_2(d_p)-R_1 (d_p)\big) e_{{m}/{2},p}^k.$ Unrolling the recursion \eqref{eqn:iter1} $m+1$ times and using $e_{0,p}^{k+1}=e_{1/2,p}^k=0$ give 
\begin{equation}\label{eqn:equation}
  e_{{m}/{2}+1,p}^{k+1} - \rho_2 (d_p) e_{{(m+1)}/{2},p}^{k+1} = \sum_{i=0}^{m}\rho_1(d_p)^{i} I_{{(m-i)}/{2},p}^k.
\end{equation}
 Hence, using the condition $|\rho_1(s)|<1$ when $s\in \mathbb{R}_+$, we deduce
\begin{align*}
  |e_{{m}/{2}+1,p}^{k+1}| &\leq |\rho_2 (d_p)| |e_{(m+1)/2,p}^{k+1}| + \sum_{i=0}^{m}|\rho_1(d_p)|^{i} |I_{(m-i)/{2},p}^k|\\
  & \leq |\rho_2 (d_p)| |e_{(m+1)/2,p}^{k+1}| + {(1-|\rho_1(d_p)|)}^{-1}\max_{1\leq m\leq 2N_c} |I_{{m}/{2},p}^k|.
\end{align*}
Since $|\rho_2(s)| <1$ for $s\in \mathbb{R}_+$, unrolling the inequality $m+1$ times gives
\begin{equation}\label{eqn:g_c}
    \begin{aligned}
|e_{{m}/{2}+1,p}^{k+1}| &\leq |\rho_{2}(d_p)|^{m+1} |e_{{1}/{2},p}^k| + \sum_{j=0}^{m} \frac{|\rho_2 (d_p)|^{j}}{1-|\rho_1(d_p)|} \max_{1\leq m\leq 2N_c} |I_{{m}/{2},p}^k|\\
&\leq {\big( (1-|\rho_1(d_p)|)(1-|\rho_2(d_p)|)\big)}^{-1}  \max_{1\leq m\leq 2N_c} |I_{{m}/{2},p}^k|\\
& \leq \frac{|r({2d_p}/{J})^J-r({2d_p}/{J})^{{J}/{2}}R_2(d_p)-R_1 (d_p)|}{(1-|\rho_1(d_p)|)(1-|\rho_2(d_p)|)}  \max_{1\leq m\leq 2N_c} |e_{{m}/{2},p}^k|\\
&=: \gamma_c (r,R_1,R_2,J,d_p) \max_{1\leq m\leq 2N_c} |e_{{m}/{2},p}^k|.
\end{aligned}
\end{equation}
Taking the maximum over $m$ for $1\leq m\leq 2N_c$  yields
\begin{equation*}
\begin{aligned}
   \max_{1\leq m\leq 2N_c} |e_{{m}/{2},p}^{k+1}| &\leq \gamma_c (r,R_1,R_2,J,d_p) \max_{1\leq m\leq 2N_c} |e_{{m}/{2},p}^k| \\
   &\leq \gamma_c (r,R_1,R_2,J,d_p) ^{k+1}\max_{1\leq m\leq 2N_c} |e_{{m}/{2},p}^0|. 
   \end{aligned}
\end{equation*}
Then the error $\|E_{{m}/{2}+1}^{k+1}\|_{L^2(\Omega)}$ is bounded by
\begin{align*}
&\max_{0\leq m\leq 2N_c}\| E_{{m}/{2}}^{k+1} \|_{L^2 \II}^2 \leq \sum_{p=1}^\infty \max_{0\leq m\leq 2N_c}|e_{{m}/{2},p}^k|^2 \\
\leq &\sup_{p\in \mathbb{N}_{+}}\gamma_c (r,R_1,R_2,J,d_p)^{2(k+1)}\sum_{p=1}^\infty\max_{1\leq m\leq 2N_c} |e_{{m}/{2},p}^0|^2.
\end{align*}
Since $d_p\in \mathbb{R}_+$ for all $p \in \mathbb{N}_{+}$, the factor can be further bounded by
\begin{equation*}
\begin{aligned}
\sup_{p\in \mathbb{N}_{+}}\gamma_c (r,R_1,R_2,J,d_p) &\leq \sup_{s\in\mathbb{R}_+}\frac{|r({2s}/{J})^J-R_2(s)r({2s}/{J})^{{J}/{2}}-R_1(s)|}{(1-|\rho_1(s)|)(1-|\rho_2(s)|)} \\
&=: \gamma^*(r,R_1,R_2,J).
\end{aligned}
\end{equation*}
Using the estimate $\max_{1\leq m\leq 2N_c}|e_{{m}/{2},p}^0|^2\leq \sum_{m=1}^{2N_c} |e_{{m}/{2},p}^0|^2$, we get
\begin{equation*}
\max_{0\leq m\leq 2N_c}\| E_{{m/2}}^{k+1} \|_{L^2 \II}^2  \leq \gamma^*(r,R_1,R_2,J)^{2(k+1)} \sum_{m=1}^{2N_c} \|E_{{m}/{2}}^0\|^2_{L^2 \II}.
\end{equation*}
By the elementary inequality $\sqrt{\sum_{m=1}^{2N_c}a_i^2} \leq \sum_{m=1}^{2N_c}|a_i|$, we obtain the desired result with  $c = \sum_{m=1}^{2N_c} \|E_{{m}/{2}}^0\|_{L^2 \II}$. 
\end{proof}

\begin{remark}\label{rmk:sharper}
The estimate in Theorem~\ref{thm:error estimation} is concise but not sharp. We can derive a sharper but more involved result. Unrolling the estimate~\eqref{eqn:equation} $m+1$ gives
    \begin{align*}
    e_{{m}/{2}+1,p}^{k+1} =\sum_{j=0}^{m} \frac{ \rho_{2} \left( d_{p} \right)^{m+1-j} - \rho_{1} \left( d_{p} \right)^{m+1-j}}{\rho_{2} \left( d_{p} \right) -\rho_{1} \left( d_{p} \right)} I_{j/2,p}^{k}.
    \end{align*}
Following the steps in Theorem~\ref{thm:error estimation},  $\|E_{m/2+1}^{k+1}\|_{L^2(\Omega)}$ can be bounded by
\begin{equation*}
    \max_{0\leq m\leq 2N_{c}} \| E_{m/2}^{k+1}\|_{L^{2}\left( \Omega \right)} \leq  {\kappa}^{\ast} \left( r,R_{1},R_{2},J,N_{c} \right)^{k+1} \sum_{m=1}^{2N_{c}} \| E_{m/2}^{0}\|_{L^{2}\left( \Omega \right)},
\end{equation*}
where $\kappa_c (r,R_1,R_2,J,s,N_c)$ and ${\kappa}^*(r,R_1,R_2,J,N_c)$ are defined by
\begin{equation}\label{eqn:k^*}
\begin{aligned}
{\kappa}_c(r,R_1,R_2,J,s,N_c) &= \frac{|r\left( 2s/J \right)^{J} -R_{2}\left( s \right) r\left( 2s/J \right)^{J/2} -R_{1}\left( s \right) |}{|\rho_{1} \left( s \right) -\rho_{2} \left( s \right) |} \\
&\quad \cdot \sum_{j=0}^{2N_{c}} |\rho_{2} \left( s \right)^{2N_c+1-j} -\rho_{1} \left( s \right)^{2N_c+1-j} |, \\ 
{\kappa}^*(r,R_1,R_2,J,N_c) &=  \sup_{s\in \mathbb{R}_{+}} \kappa_c (r,R_1,R_2,J,s,N_c). 
\end{aligned}
\end{equation}
Nevertheless, the convergence factor $\gamma^{*}$ is more amenable to optimization in Section~\ref{sec:optim}. 

\end{remark}

\begin{remark}
$\rho_1(s)$ and $\rho_2(s)$ are two roots of the equation
$  z^2 - R_1(s) z - R_2 (s)=0$, given by $
\rho_{1,2}(s)={\big(R_1(s)\pm\sqrt{(R_1(s))^2+4R_2(s)}\big)}/{2}$.
Note that they may contain imaginary parts since $(R_1(s))^2+4R_2(s)$ can be negative when $s\in \mathbb{R}_+.$ The conditions $|\rho_1(s)|,|\rho_2|(s) <1$ imply that the two-step CP is stable. \end{remark}

Now we illustrate Theorem~\ref{thm:error estimation} with the BDF2 scheme as the CP, with 
\begin{equation}\label{eqn:BDF2}
    R_1(s)=\frac{-\frac13}{1+\frac23 s} \quad \mbox{and}\quad R_2(s)=\frac{\frac43}{1+\frac23s}.
\end{equation}
We take the three-stage Radau IIA method as the FP, whose stability function $r(s)$ is given by
\begin{equation}\label{eqn:RIIA3}
    r(s) = \frac{1 - \frac{2}{5}s + \frac{1}{20}s^2}{1 + \frac{3}{5}s + \frac{3}{20}s^2 + \frac{1}{60}s^3}.
\end{equation}
The plots of $|\rho_1(s)|, |\rho_2(s)|, \gamma_c(r,R_1,R_2,J,s)$~\eqref{eqn:g_c} and $\kappa_c (r,R_1,R_2,J,s,10^3)$~\eqref{eqn:k^*} are given in Fig.~\ref{fig:rho_BDF2}. Fig.~\ref{fig:rho_BDF2}~(a) shows that $|\rho_1 (s)|=1+\mathcal{O}(s)$ as $s\rightarrow 0^+$ and $|\rho_2(s)|\leq C<1$ for all $s\in \mathbb{{R}}_+$. Fig.~\ref{fig:rho_BDF2}~(b) indicates that the convergence factor $\gamma_c (r,R_1,R_2, J,s)$ varies little for large $J$ and is close to the limit
\begin{equation}\label{eqn:g_e}
    \gamma_e(R_1, R_2,s) = \frac{|e^{-2s} - R_2 (s)e^{-s} - R_1(s)|}{(1-|\rho_1 (s)|) (1-|\rho_2 (s)|)}.
\end{equation}
Theorem~\ref{thm:reduced conv.} below indicates $\sup_{s\in \mathbb{R}_+} |\gamma_c(r,R_1,R_2,J,s) - \gamma_e (R_1,R_2,s)| \sim \mathcal{O}(J^{-5})$ for the three-stage Radau IIA as the FP. Fig.~\ref{fig:rho_BDF2}~(c) shows that  $\kappa^\ast$ in~\eqref{eqn:k^*} also varies very little for large $J$ and is close to the limit 
\begin{equation}\label{eqn:k_e}
{\kappa}_e(R_1,R_2,s,N_c) = \frac{|e^{-2s} -R_{2}\left( s \right) e^{-s} -R_{1}\left( s \right) |}{|\rho_{1} \left( s \right) -\rho_{2} \left( s \right) |} \sum_{j=0}^{2N_{c}} |\rho_{2} \left( s \right)^{2N_c+1-j} -\rho_{1} \left( s \right)^{2N_c+1-j} |.
\end{equation}
The numerical results in Section~\ref{sec:linear} indicate that the estimate $\kappa^\ast$ in~\eqref{eqn:k^*} is sharp.

\begin{figure}[hbt!]
\begin{minipage}[t]{0.33\textwidth}
\includegraphics[width=\textwidth,trim={0.1cm 0.2cm 1cm 1cm},clip]{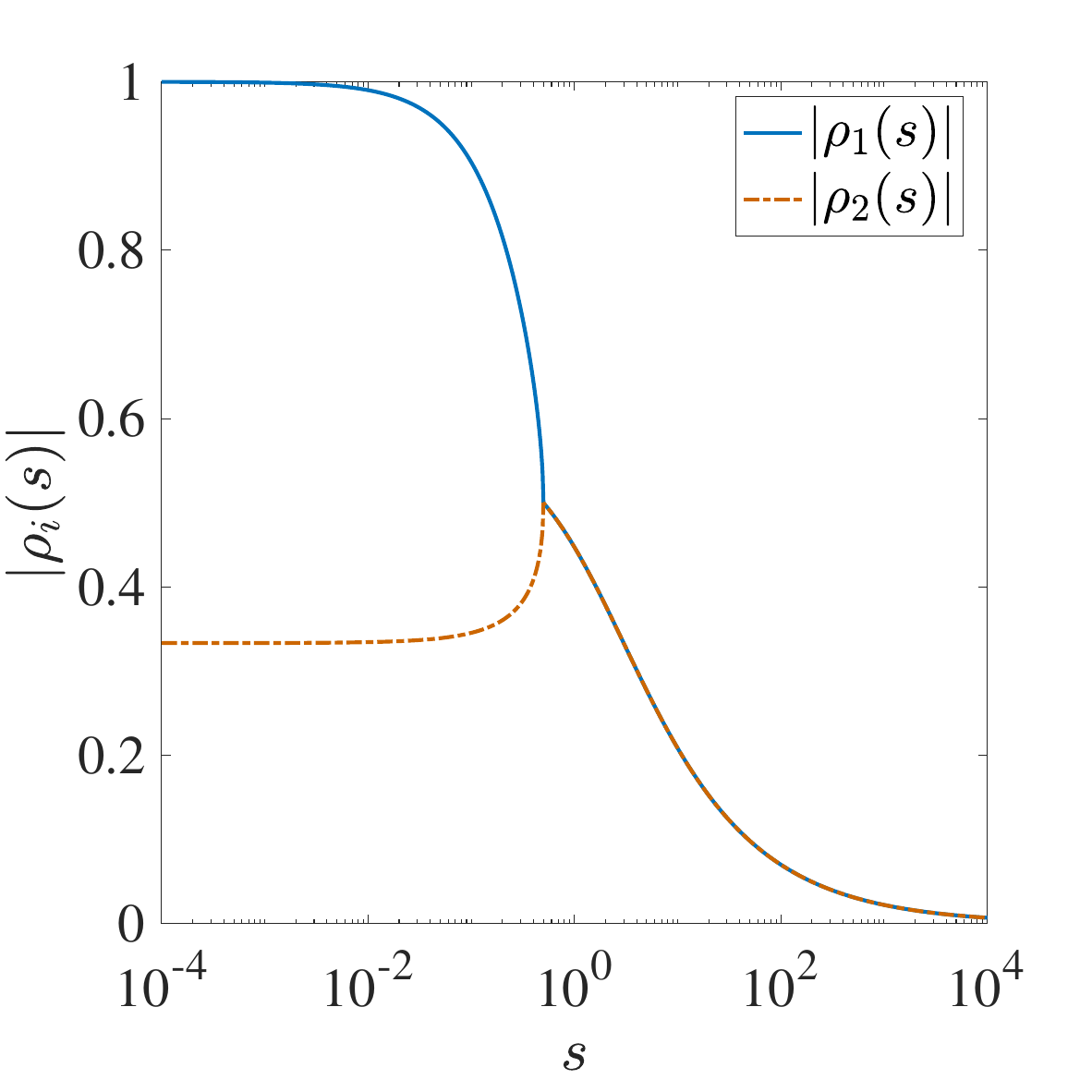}

\centering
(a)
\end{minipage}\hfill
\begin{minipage}[t]{0.33\textwidth}
\includegraphics[width=\textwidth,trim={0.1cm 0.2cm 1cm 1cm},clip]{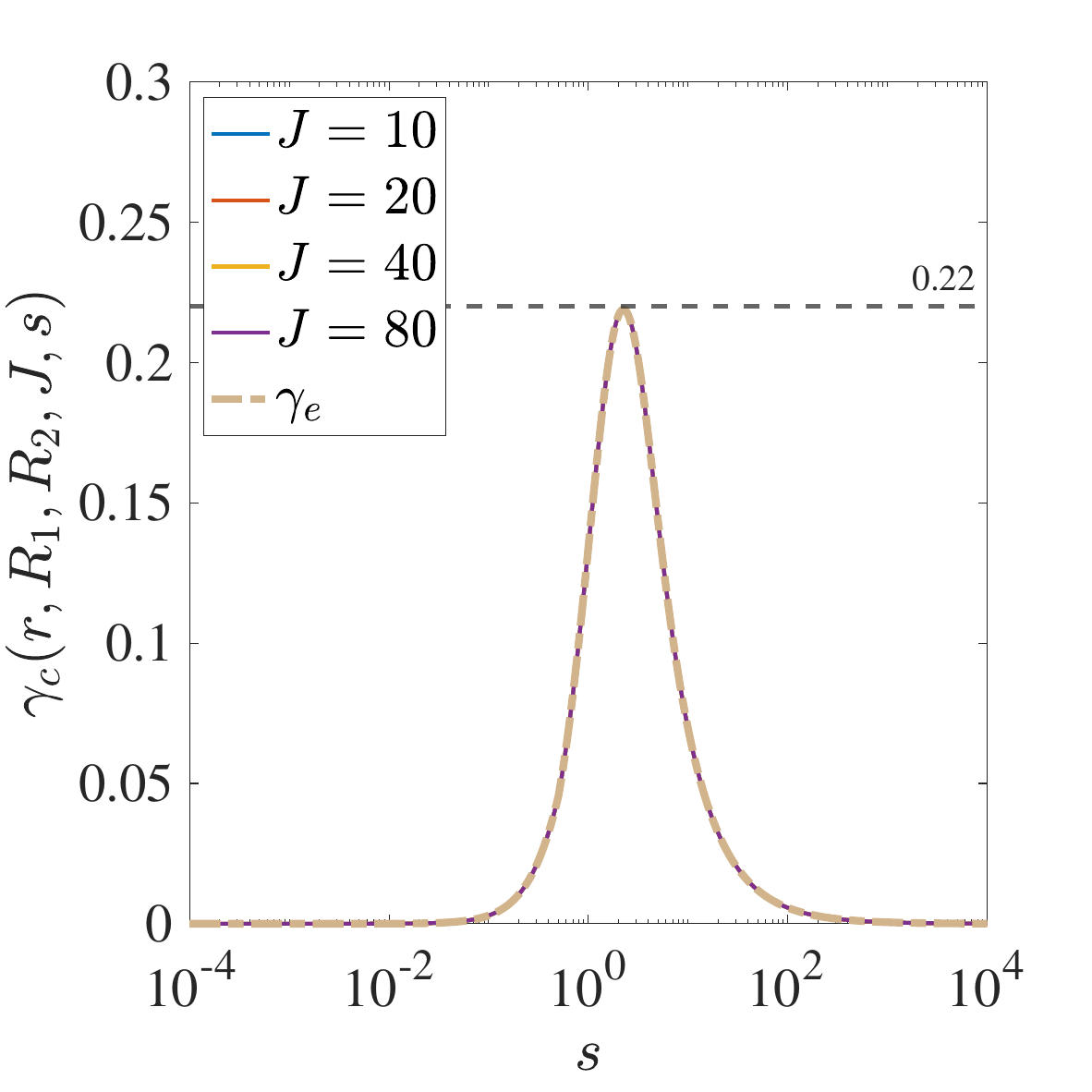}

\centering
(b)
\end{minipage}\hfill
\begin{minipage}[t]{0.33\textwidth}
\includegraphics[width=\textwidth,trim={0.1cm 0.2cm 1cm 1cm},clip]{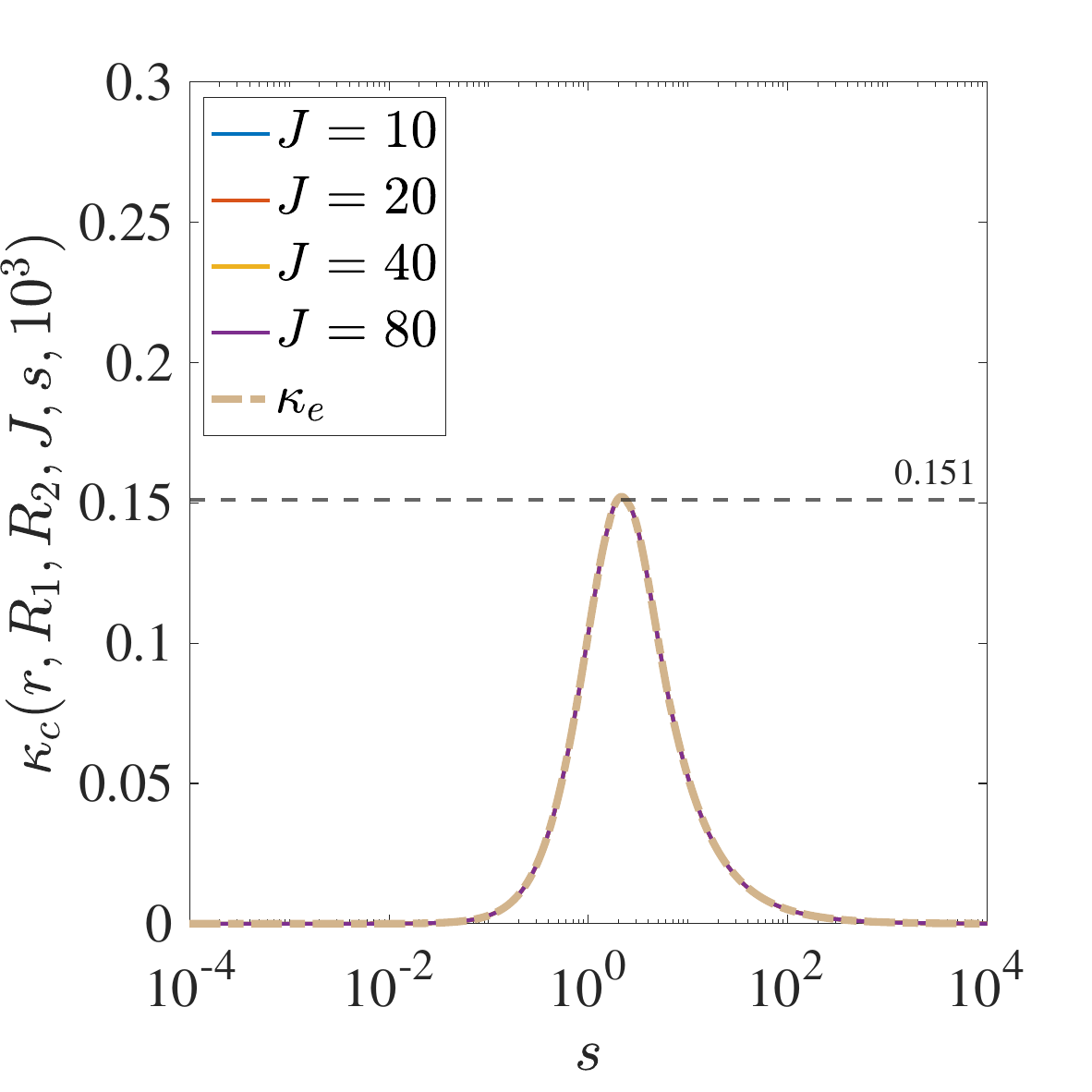}

\centering
(c)
\end{minipage}
\caption{With BDF2 as the CP and three-stage Radau IIA as the FP: (a) $|\rho_i (s)|,~i=1,2$; (b) $\gamma_c(r,R_1,R_2,J,s)$ for $J\in \{10,20,40,80\}$ and $\gamma_e$ in~\eqref{eqn:g_e}, and (c) $\kappa_c(r,R_1,R_2,J,s,10^3)$ for $J\in \{10,20,40,80\}$ and $\kappa_e$ in~\eqref{eqn:k_e}.}
    \label{fig:rho_BDF2}
\end{figure}

\subsection{Reduced convergence factor}
Motivated by Fig.~\ref{fig:rho_BDF2}(b)-(c), We define factor $\gamma_e^*$ and $\kappa_e^*$ by 
\begin{equation}\label{eqn:reduced conv.}
\begin{aligned}
\gamma_e^*(R_1,R_2) &= \sup_{s\in \mathbb{R}_+} |\gamma_e (R_1,R_2,s)|,\\
\kappa_e^*(R_1,R_2,N_c) &= \sup_{s\in \mathbb{R}_+} |\kappa_e (R_1,R_2,s,N_c)|,
\end{aligned}
\end{equation}
and discuss the connection between $\gamma^*(r,R_1,R_2,J)$ and $\gamma_e^* (R_1,R_2)$, and that between $\kappa^*(r,R_1,R_2,J,N_c)$ and $\kappa_e^* (R_1,R_2,N_c)$. We impose the following conditions on the CPs and FPs.     The assumption on $r$ can be satisfied if the FP is a $q$-th order L-stable solver. The assumption on $\rho_1$ and $\rho_2$ is motivated by Fig.~\ref{fig:rho_BDF2}.

\begin{assumption}\label{assum:s0}
There exists $s_0>0$ such that $R_1$, $R_2$, $r$, $\gamma_c$ and $\gamma_e$ satisfy the following conditions:
\begin{itemize}
\item[{\rm(i)}] When $s\in (0,s_0)$: $0<c_1 s\leq (1-|\rho_1 (s)|)(1-|\rho_2(s)|)$ and $|r(s)-e^{-s}|\leq c_2 s^{q+1}$ with $c_1>0$ and $q\geq 2$.
\item[{\rm(ii)}] $\gamma_c (r,R_1,R_2,J,s)$ and $\gamma_e (R_1,R_2,s)$ attain their supremum when $s\in (0,s_0)$ for all $J\in 2\mathbb{N}_{+}$.
\item[{\rm(iii)}] $\kappa_c (r,R_1,R_2,J,s,N_c)$ and $\kappa_e (R_1,R_2,s,N_c)$ attain their supremum when $s\in (0,s_0)$ for all $J\in 2\mathbb{N}_{+}$.
\end{itemize}
\end{assumption}

The next theorem shows how $\gamma^* (r,R_1,R_2, J)$ approximates $\gamma_e^* (R_1,R_2)$. 

\begin{theorem}\label{thm:reduced conv.}
Let $R_1$,~$R_2$,~$r,~\gamma_c$, and $\gamma_e$ satisfy Assumption~\ref{assum:s0} (i) and (ii). Then there exists $C>0$ independent of $J$ such that 
\begin{equation*}
|\gamma^* (r,R_1,R_2,J)-\gamma_e^* (R_1,R_2)| \leq {C}{J^{-q}}.
\end{equation*}
\end{theorem}

\begin{proof}
In view of Assumption~\ref{assum:s0} (ii), it suffices to analyze the case $s\in (0,s_0)$. By Assumption~\ref{assum:s0} (i), we have
\begin{align*}
&|\gamma_c (r,R_1,R_2,J,s)-\gamma_e (R_1,R_2,s)|\\
\leq& \frac{|e^{-2s}-r\left( {2s/J} \right)^{J} |}{\left( 1-|\rho_{1} \left( s \right) | \right) \left( 1-|\rho_{2} \left( s \right) | \right)} +\frac{|R_{2}\left( s \right) ||e^{-s}-r\left( {2s/J} \right)^{{J/2}} |}{\left( 1-|\rho_{1} \left( s \right) | \right) \left( 1-|\rho_{2} \left( s \right) | \right)}\\
\leq& \frac{|e^{-2s}-r\left( {2s/J} \right)^{J} |}{c_1 s} + \frac{C|e^{-s}-r\left( {2s/J} \right)^{{J/2}} |}{c_1s}=:\frac{{\rm I_1}}{c_1s} + \frac{C~{\rm I_2}}{c_1s}.
\end{align*}
The first factor $\rm I_1$ can be bounded as
\begin{align*}
{\rm I_1}\leq |e^{-{2s}/J} -r({2s}/{J})|~|\sum_{i=0}^{J-2} (e^{-{2s}/{J}})^{(J-1)-i} r({2s}/{J})^i|\leq c_2 {(2s)^{q+1}}{J^{-q}}.
\end{align*}
Similarly, the factor ${\rm I}_2$ can be bounded as ${\rm I_2} \leq c_2 {(2s)^{q+1}}/(2J^q).$ Then taking the supremum over $s\in(0,s_0)$ gives
\begin{equation*}
\sup_{s\in (0,s_0)}|\gamma_c (r,R_1,R_2,J,s)-\gamma_e (R_1,R_2,s)| \leq C {s_0^{q}}{J^{-q}} \leq C J^{-q}.
\end{equation*}

\end{proof}

We illustrate Theorem \ref{thm:reduced conv.} in Fig.~\ref{fig:rho_Radau_LIIIC} (c) with three FPs.  The next corollary shows how $\kappa^* (r, R_1, R_2, J, N_c)$ approximates $\kappa_e^*$. The proof is similar to Theorem~\ref{thm:reduced conv.}, and hence omitted.
\begin{corollary}
    Let $R_1$,~$R_2$,~$r$,~ $\kappa_c$ and $\kappa_e$ satisfy Assumption~\ref{assum:s0} (i) and (iii). Then there exists $C>0$ independent of $J$ and $N_c$ such that 
\begin{equation*}
|\kappa^* (r,R_1,R_2,J,N_c)-\kappa_e^* (R_1,R_2,N_c)| \leq {C}{J^{-q}}.
\end{equation*}
\end{corollary}

section{Optimizing two-step CPs}\label{sec:optim}
Now we optimize a two-step CP to expedite the convergence of Algorithm~\ref{alg:two-step parareal}, based on $\gamma_e^*$ in~\eqref{eqn:reduced conv.}. The approach has two distinct features. First, $\gamma_e^*$ is independent of the specific structure of the FP and the coarsening factor $J$. Theorem~\ref{thm:reduced conv.} shows that when the FP is of high order and $J$ is not small, the actual convergence factor approximates well the reduced one. Second, $\gamma_e^*$ is amenable to optimization.

\subsection{Optimization procedure}
We define two parametric rational functions.
\begin{equation*}
    R_1(s,\bt) = \frac{a_1+a_2 s}{1+e^{b_1}s}\quad\mbox{and}\quad R_2 (s,\bt) = \frac{(1-a_1)+c_2 s}{1+e^{b_1} s},
\end{equation*}
with the parameter vector $\bt = (a_1,a_2,b_1,c_2)$. The parametric two-step CP reads 
\begin{equation}\label{eqn:M-OCP}
\begin{aligned}
    \G^*_{\bt}(T_n,\tau,v_1,v_2) &= R_1(\tau A,\bt) v_1 + R_2 (\tau A,\bt)v_2 \\
    &\quad + \tau \big( -a_2 f(T_n -\tau )-c_2 f(T_n) + e^{b_1} f(T_n + \tau) \big), 
\end{aligned}
\end{equation}
The scheme is consistent, cf. the first condition in~\eqref{eqn:order_k_step}. Then $\rho_1(s,\bt)$ and $\rho_2(s,\bt)$ are given in \eqref{eqn:rho}.
To allow greater flexibility in parameter optimization, we do not impose order conditions. The optimized $R_1(s,\hat{\bt})$ and $R_2(s,\hat{\bt})$ should minimize $\gamma^*_e$ \eqref{eqn:reduced conv.}. To optimize  the CP $\G_{\bt}^*$, the supremum in $\gamma_e^*$ over $s \in \mathbb{R}_+$ is approximated using a finite sample set $\Lambda$ consisting of $N_\Lambda$ uniformly spaced points in $[0, \lambda_{\max}]$, with the parameters $N_\Lambda$ and $\lambda_{\max}$ chosen based on $A$. Then the optimization problem reads
\begin{equation}\label{eqn:opt problem}
    \min_{\bt} \gamma_e^* (R_1,R_2) \quad \text{s.t.} \quad |\rho_1 (s,\bt)|,|\rho_2(s,\bt)| <1, ~s >0.
\end{equation}
The constraints ensure the stability of $\G_{\bt}^*$. To enforce the constraints strictly, we employ the barrier method \cite[Chapter 13]{luenberger1984linear}, using the barrier function
\begin{equation*}
    \L_{\rm b} (\bt) = \frac{1}{N_\Lambda} \sum_{s\in \Lambda} \log(1-\rho_1(s,\bt)^2) + \log(1-\rho_2(s,\bt)^2).
\end{equation*}
Then the total loss $\L_{\rho}$ is defined by
\begin{equation}\label{eqn:loss_mu}
    \L_\mu (\bt) = \L_{\rm s} (\bt) - \mu \cdot \L_{\rm b} (\bt),\quad \mbox{with }\L_{\rm s} (\bt) = \sup_{s\in \Lambda} \frac{|e^{-2s} - R_2(s,\bt)e^{-s}-R_1(s,\bt)|}{(1-|\rho_1(s,\bt)|)(1-|\rho_2 (s,\bt)|)},
\end{equation}
where $\mu>0$ is a weight. To promote the convergence of the barrier method for problem~\eqref{eqn:opt problem}, we initialize $\mu$ with $\mu_0$ and geometrically decrease its value as $\mu_{i+1} = \sigma \mu_i$ for some $\sigma \in (0, 1)$. Let $\bt_k^*$ denote a minimizer of the loss $\L_{\mu_k} (\bt)$. Any limit point of the sequence $\{\bt_k^* \}$ solves problem~\eqref{eqn:opt problem} \cite[Chapter 13]{luenberger1984linear}. The full procedure for obtaining the OCP $\G_{\bt}^*$ is given in Algorithm~\ref{algo:M-OCP}. To minimize the loss $\L_{\mu}$, we adopt subgradient descent; see \cite{goffin1977convergence, bianchi2022convergence, scaman2022convergence, khaled2020better} for the convergence analyses.
In practice, the stopping criterion at step 8 is fulfilled when the norm of the loss gradient drops below a given tolerance.

\begin{algorithm}[hbt!]
\caption{Optimize the two-step CP }
\begin{algorithmic}[1]
\State Specify parameters $N_\Lambda, \lambda_{\max}$ and $\sigma$, and initialize $\bt$ with $\bt_0$.
\For {$i=0,1,\ldots$}
\State $\mu_{i+1}=  \sigma \mu_{i}.$
\For {$k=0,1,\ldots,K$}
\State Construct the loss  $\L_{\mu_i}(\bt) = \L_{\rm s}(\bt) - \mu_{i} \cdot \L_{\rm b}(\bt)${.}
\State Minimize the loss $\mathcal{L}_{\mu_i}(\bt)$  and update $\bt$.
\EndFor
\State Check the stopping criterion.
	\EndFor
	\Return $\bt$.
\end{algorithmic}
\label{algo:M-OCP}
\end{algorithm}

\subsection{Optimized two-step coarse propagators}
Algorithm~\ref{algo:M-OCP} yields an optimized two-step CP~(O2CP) 
\begin{equation}\label{eqn:TS OCP_E}
R_1(s,\hat{\bt}) = \frac{0.02178-0.00047s}{1+0.56380s}\quad\mbox{and}\quad R_2(s,\hat{\bt}) = \frac{0.97822 - 0.46300s}{1+0.56380s}.
\end{equation}
The reduced convergence factor $\gamma_e^* (R_1(s,\hat{\bt}),R_2(s,\hat{\bt}))$ is about $0.0064$. It can be verified that the {O2CP} satisfies Assumption~\ref{assum:s0}. We illustrate $|\rho_1(s)|,|\rho_2(s)|$ and $\gamma_c$ using the three-stage Radau IIA as the FP in Fig.~\ref{fig:rho_OCP}. The function $|\rho_2 (s)|$ is increasing but stays bounded within $[0,0.9)$. $\gamma_e(R_1,R_2,s)$ exhibits several local maxima, with a maximum value around $0.0064$. The factor $\kappa_e^*$ is slightly smaller than $\gamma_e^*$. Figs.~\ref{fig:rho_OCP} and~\ref{fig:rho_BDF2} (CP is BDF2) indicate that, for each $J$, the convergence factor $\gamma_c$ of the O2CP is much smaller than that of BDF2. The O2CP is slightly more expensive than BDF2 (requiring one more matrix-vector product per step), but it trades off well between fast convergence and high efficiency.

\begin{figure}[hbt!]
\begin{minipage}[t]{0.33\textwidth}
\includegraphics[width=\textwidth,trim={0.1cm 0.2cm 1cm 1cm},clip]{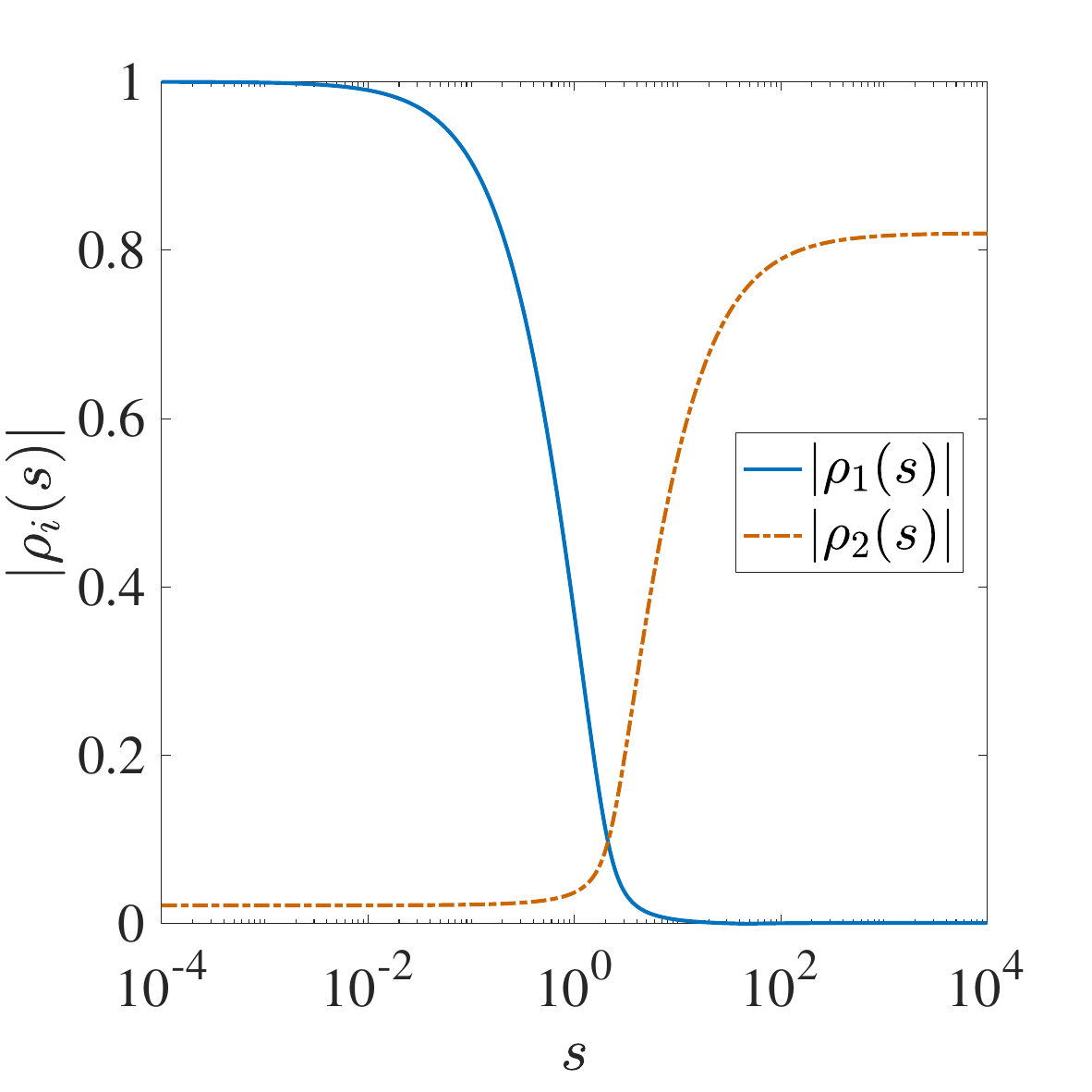}

\centering
(a)
\end{minipage}\hfill
\begin{minipage}[t]{0.33\textwidth}
\includegraphics[width=\textwidth,trim={0.1cm 0.2cm 1cm 1cm},clip]{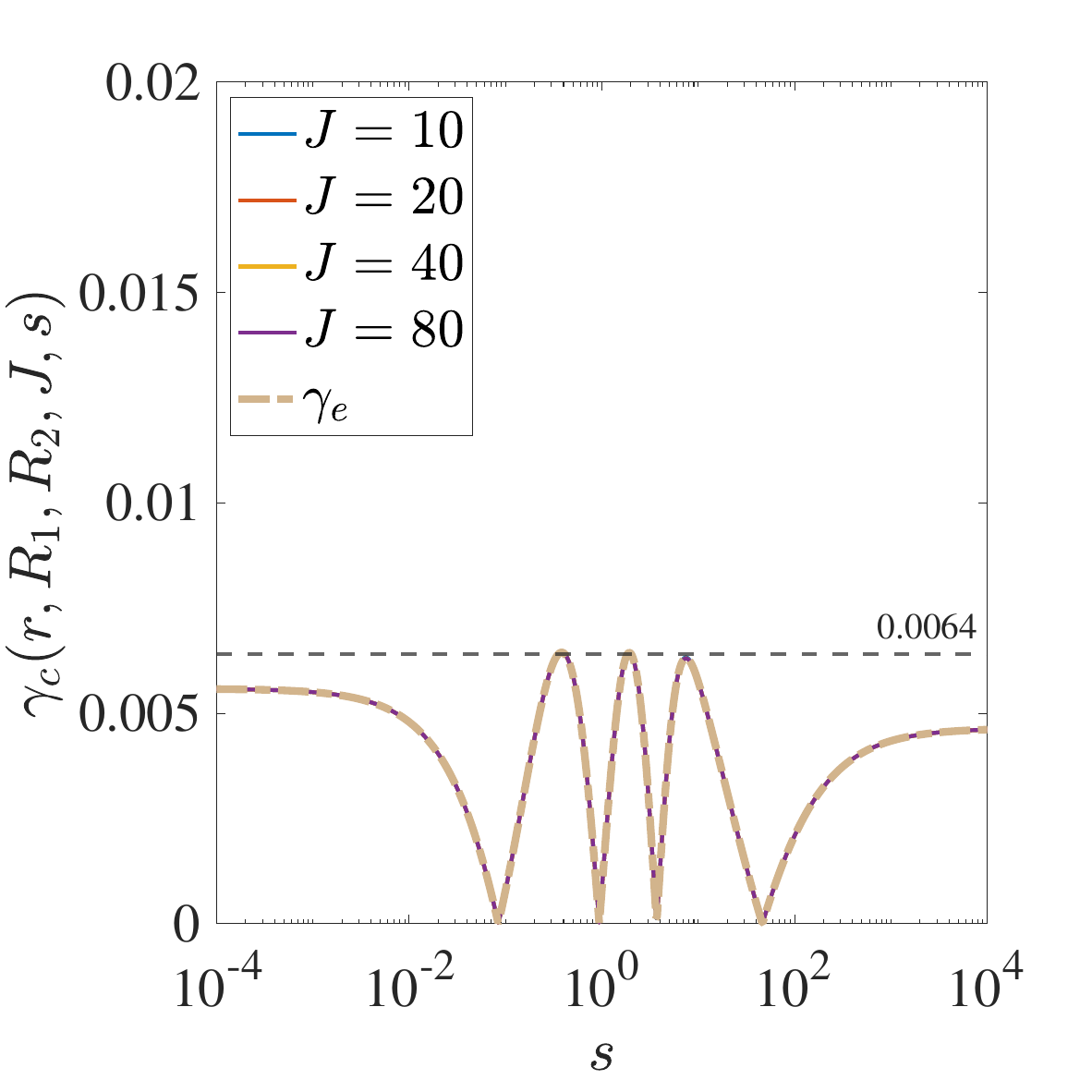}

\centering
(b)
\end{minipage}\hfill
\begin{minipage}[t]{0.33\textwidth}
\includegraphics[width=\textwidth,trim={0.1cm 0.2cm 1cm 1cm},clip]{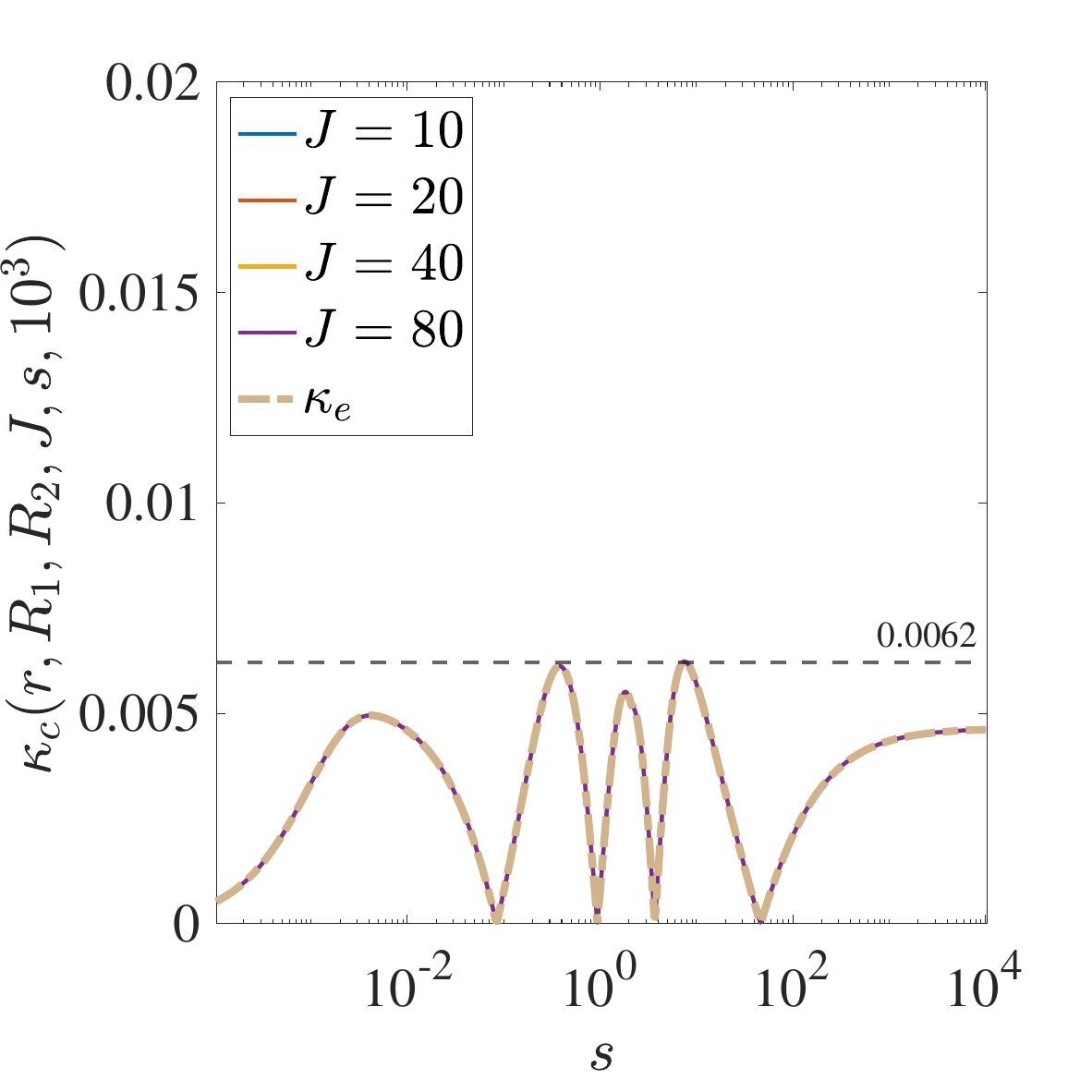}

\centering
(c)
\end{minipage}
    \caption{With the O2CP~\eqref{eqn:TS OCP_E} as the CP and three-stage Radau IIA as the FP: (a) $|\rho_i (s)|~i=1,2$, (b) $\gamma_c(r,R_1,R_2,J,s)$ for $J\in \{10,20,40,80\}$ and $\gamma_e$ in~\eqref{eqn:g_e}, and (c) $\kappa_c(r,R_1,R_2,J,s,10^3)$ for $J\in \{10,20,40,80\}$ and $\kappa_e$ in~\eqref{eqn:k_e}.\label{fig:rho_OCP}}
\end{figure}

Algorithm~\ref{algo:M-OCP} adapts the CP to the FP and the coarsening factor $J$ so as to achieve fast convergence. We evaluate the O2CP using two-stage Radau IIA and three-stage Lobatto IIIC as FPs in Fig.~\ref{fig:rho_Radau_LIIIC} (a)-(b). The behavior of $\gamma_e$ is similar to that for three-stage Radau IIA in Fig.~\ref{fig:rho_OCP}. Fig.~\ref{fig:rho_Radau_LIIIC} (c) shows Theorem~\ref{thm:reduced conv.} for the three FPs.

\begin{figure}[hbt!]
\begin{minipage}[t]{0.33\textwidth}
\includegraphics[width=\textwidth,trim={0.1cm 0.2cm 1cm 1cm},clip]{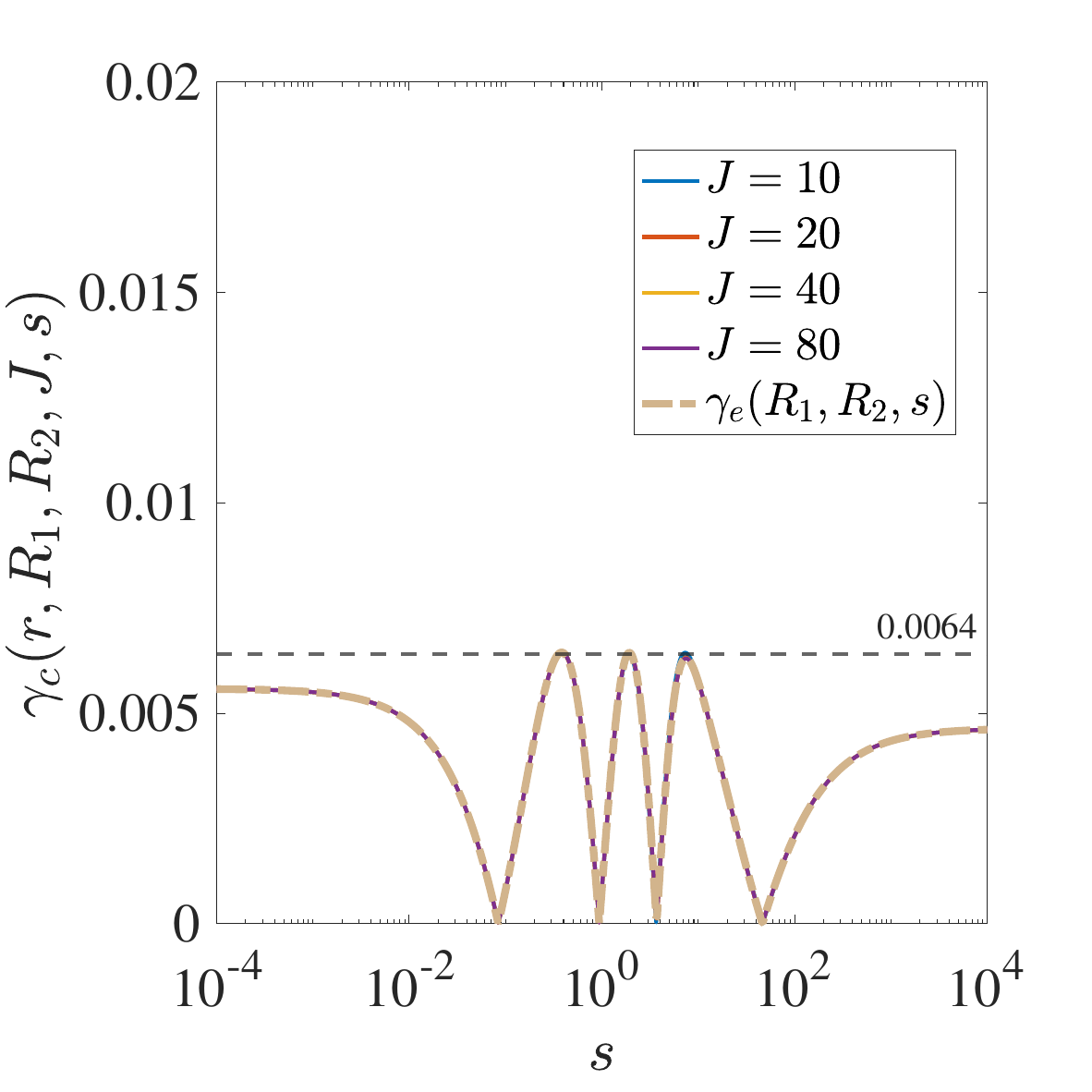}

\centering
(a)
\end{minipage}\hfill
\begin{minipage}[t]{0.33\textwidth}
\includegraphics[width=\textwidth,trim={0.1cm 0.2cm 1cm 1cm},clip]{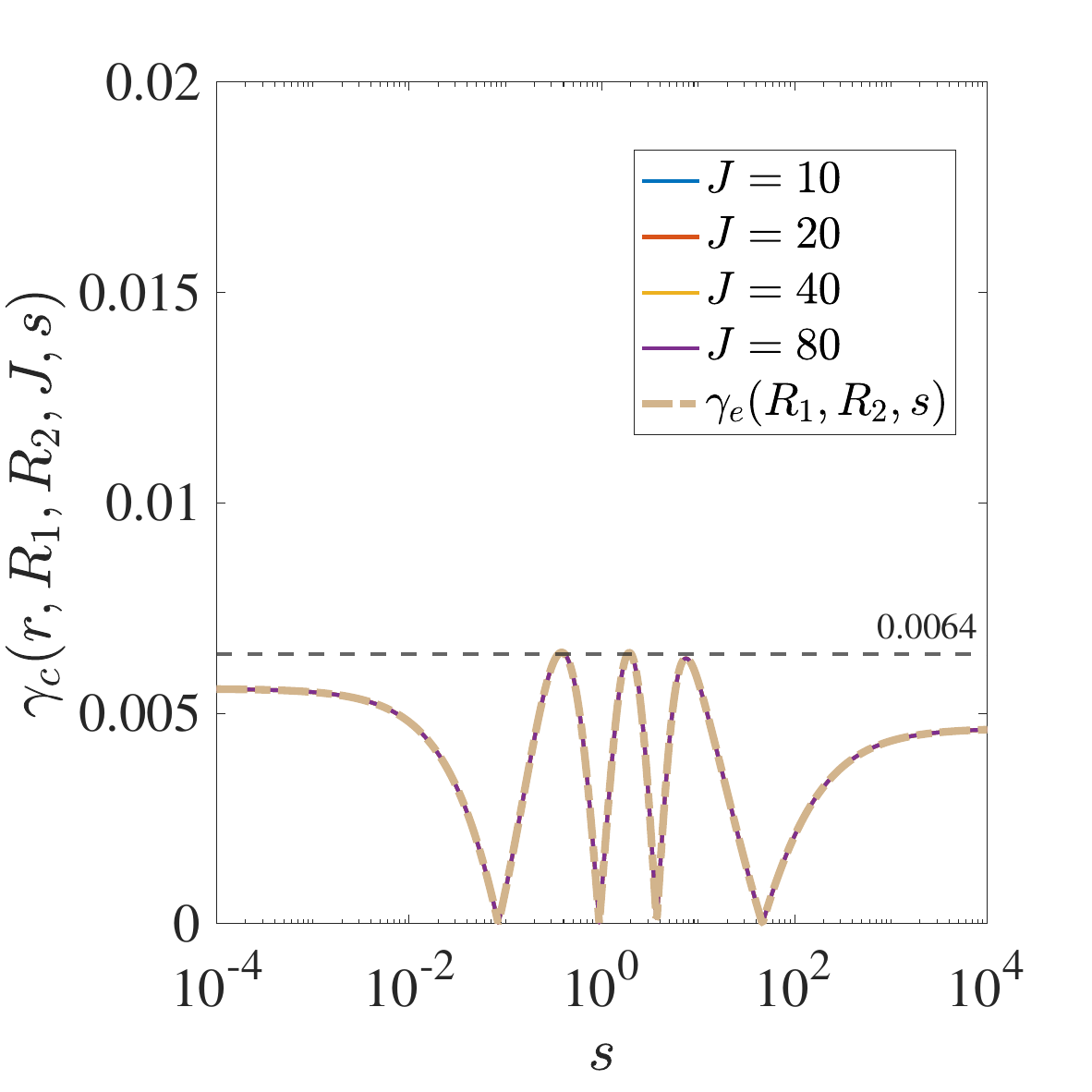}

\centering
(b)
\end{minipage}\hfill
\begin{minipage}[t]{0.33\textwidth}
\includegraphics[width=\textwidth,trim={0.1cm 0.2cm 1cm 1cm},clip]{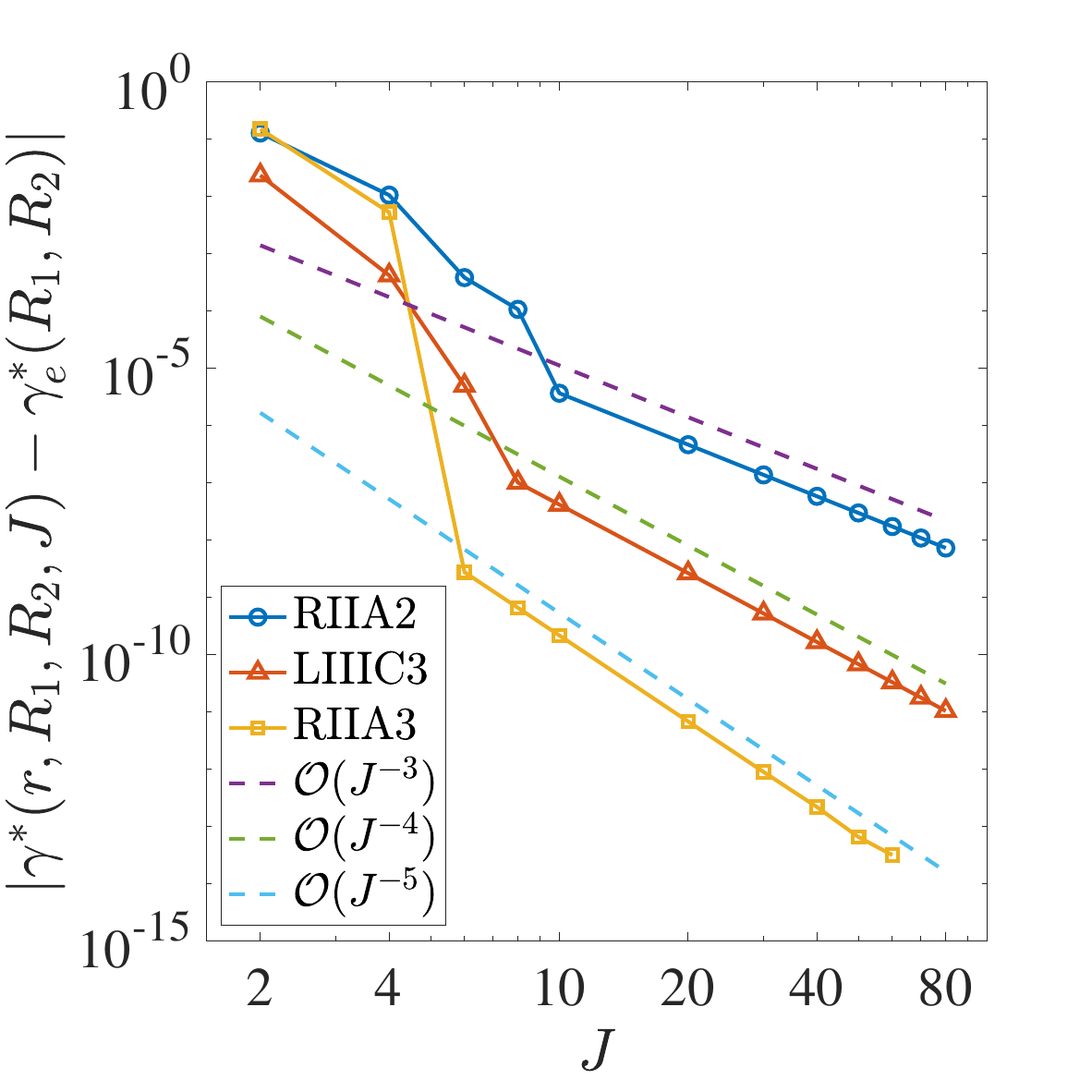}

\centering
(c)
\end{minipage}
\caption{$\gamma_c(r,R_1,R_2,J,s)$ for $J\in \{10,20,40,80\}$ with the O2CP for two-stage Radau IIA (a) and three-stage Lobatto IIC (b) as the FP; (c) illustration of the bound on $|\gamma^*(r,R_1,R_2,J)-\gamma_e^*(R_1,R_2)|$ in Theorem~\ref{thm:reduced conv.} for three FPs: two-stage Radau IIA (RIIA2), three-stage Lobatto IIIC (LIIIC3), and three-stage Radau IIA (RIIA3).}
    \label{fig:rho_Radau_LIIIC}
\end{figure}

\subsection{Properties of the O2CP}
Now we examine the stability, order condition and the convergence function $\gamma_e$ of the O2CP. First, it is A-stable and consistent.
\begin{proposition}
The O2CP is A-stable and consistent. 
\end{proposition}
\begin{proof}
For the A-stability, applying the O2CP to $y'+\lambda y=0$ yields
    \begin{equation*}
        y_{n+2}+\alpha_{0} y_{n}+\alpha_{1} y_{n+1}=-\Delta T \lambda\left( \beta_{0} y_{n}+\beta_{1} y_{n+1}+\beta_{2} y_{n+2} \right),
\end{equation*}
with $\alpha_i,~\beta_i$ defined in~\eqref{eqn:R_a_b}.
Let $\mu=\Delta T \lambda$ and $y_n = \zeta^n$. Then we obtain the characteristic equation
\begin{equation*}
(1+\beta_{2} \mu)\zeta^{2} +(\alpha_{1} +\beta_{1} \mu )\zeta +\left( \alpha_{0} +\beta_{0} \mu \right) =-\varrho \left( \zeta \right) +\mu \sigma \left( \zeta \right) =0.
\end{equation*}
The root locus curve is given by
\begin{equation*}
    \mu =\frac{\varrho \left( \zeta \right)}{\sigma \left( \zeta \right)} =\frac{- \zeta^{2} -\alpha_{1} \zeta -\alpha_{0}}{\beta_{2} \zeta^{2} +\beta_{1} \zeta +\beta_{0}} ,\quad \mbox{with }\zeta =e^{i\theta}.
\end{equation*}
To show $\Re(\mu) \leq 0$, let $\varrho = \varrho_r + i \varrho_i$ and $\sigma = \sigma_r + i \sigma_i$, with the subscripts $r$ and $i$ indicating the real and imaginary parts, respectively. Then, $\Re \left( \mu \right) =\Re \left( {(\varrho_{r} \sigma_{r} +\varrho_{i} \sigma_{i})}/{|\sigma |^{2}} \right)$. Let $f\left( \theta \right) =\varrho_{r} \sigma_{r} +\varrho_{i} \sigma_{i}$. Then we have 
\begin{equation*}
\begin{aligned}
f\left( \theta \right) &=\left( -2\alpha_{0} \beta_{2} -2\beta_{0} \right) \cos^{2} \theta -\left( \alpha_{0} \beta_{1} +\alpha_{1} \beta_{0} +\alpha_{1} \beta_{2} +\beta_{1} \right) \cos \theta \\
&\quad +\left( \alpha_{0} \beta_{2} +\beta_{0} -\alpha_{0} \beta_{0} -\alpha_{1} \beta_{1} -\beta_{2} \right) .
    \end{aligned}
\end{equation*}
Substituting the coefficients from~\eqref{eqn:TS OCP_E} confirms that $f(\theta) \leq 0$ for $\theta \in (0, 2\pi)$. Since the mappings are conformal and the positive real axis belongs to the stability region, the right half complex plane $\mathbb{C}_+$ is contained in the stability region.
The consistency follows from the order conditions in~\eqref{eqn:order_k_step}.
\end{proof}

The stability regions of BDF2 and  O2CP are shown in Fig.~\ref{fig:stability}. Both are A-stable, but BDF2 has a larger stability region. The O2CP does not satisfy the first-order condition in~\eqref{eqn:order_k_step}, but BDF2 is of second order. Fig.~\ref{fig:contour} presents $\gamma_e$ for two solvers: $\gamma_e$ for the O2CP is much smaller than that of BDF2, particularly when $\Im(s)$ is small. For both methods, $\gamma_e$ exceeds $1$ when $\Re(s) = 0$. Fig.~\ref{fig:contour_new} shows $\kappa_e$ in~\eqref{eqn:k_e} for two solvers. Indeed, $\kappa_e$ provides a sharper estimate than $\gamma_e$, cf. Remark~\ref{rmk:sharper}.

\begin{figure}[htbp!]
    \centering
\includegraphics[width=0.98\textwidth,trim={0cm 1cm 0cm 1.5cm},clip]{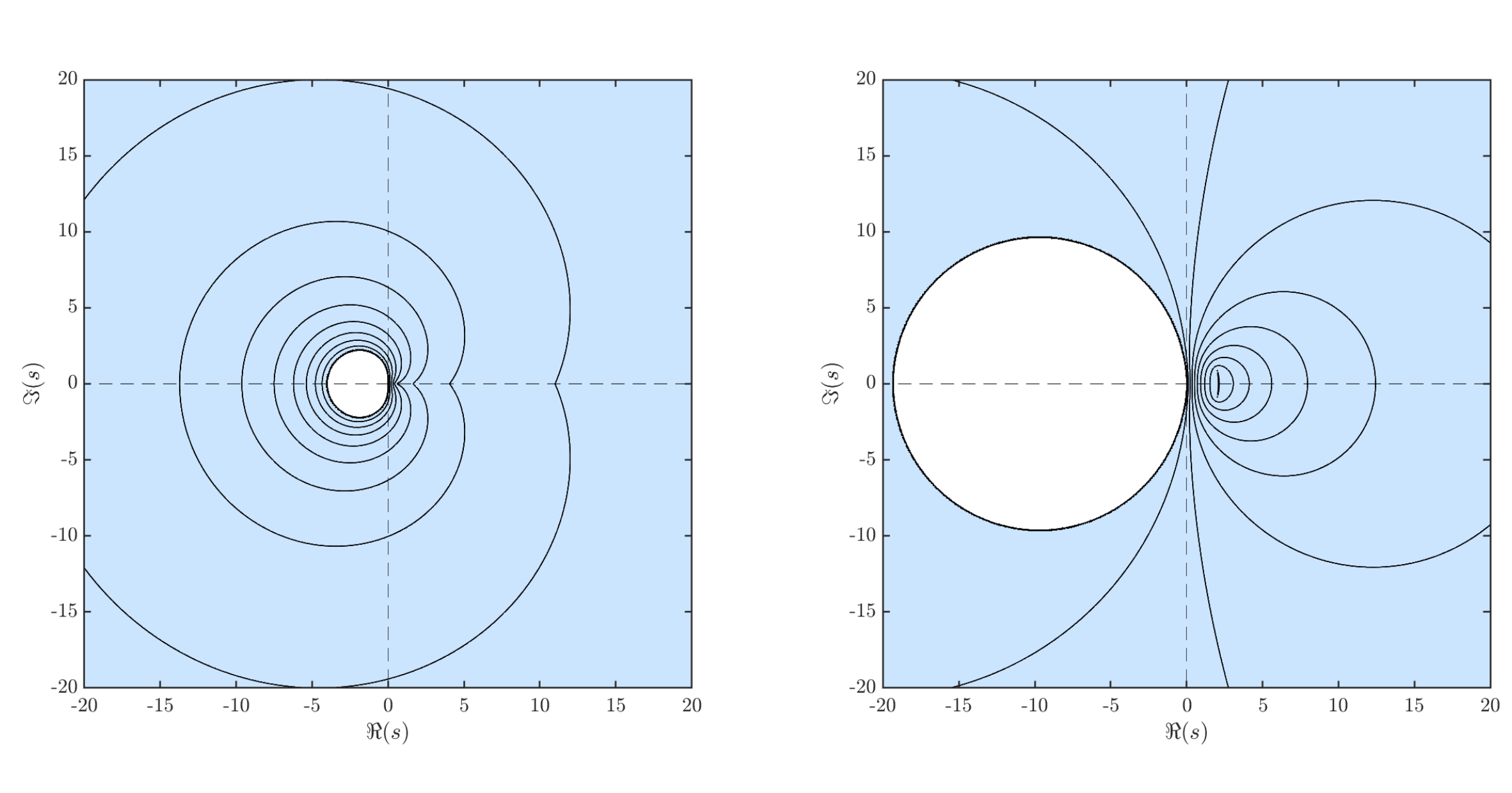}
    
    \begin{minipage}[t]{0.48\textwidth}
        \centering
        (a)
    \end{minipage}\hfill
    \begin{minipage}[t]{0.48\textwidth}
        \centering
        (b)
    \end{minipage}

    \caption{The stability region (blue) of (a) the BDF2 and (b) the O2CP. }
    \label{fig:stability}
\end{figure}

\begin{figure}[htbp!]
    \centering
\includegraphics[width=0.98\textwidth,trim={0cm 1cm 0cm 1.5cm},clip]{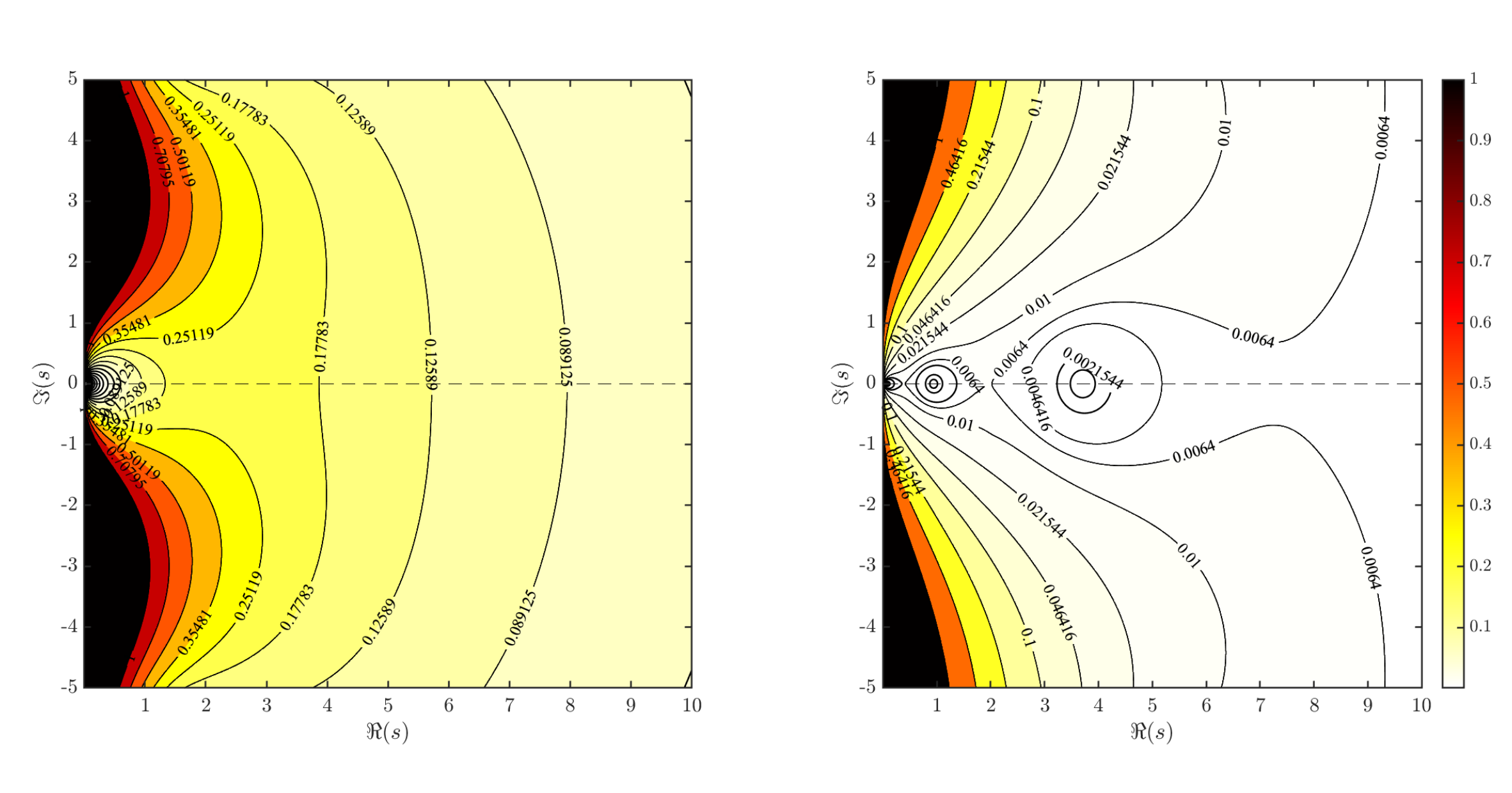}
    
    \begin{minipage}[t]{0.48\textwidth}
        \centering
        (a)
    \end{minipage}\hfill
    \begin{minipage}[t]{0.48\textwidth}
        \centering
        (b)
    \end{minipage}

    \caption{The contour maps of $\gamma_e (R_1,R_2,s)$ for (a) the BDF2 and (b) the O2CP. }
    \label{fig:contour}
\end{figure}

\begin{figure}[htbp!]
    \centering
\includegraphics[width=0.98\textwidth,trim={0cm 1cm 0cm 1.5cm},clip]{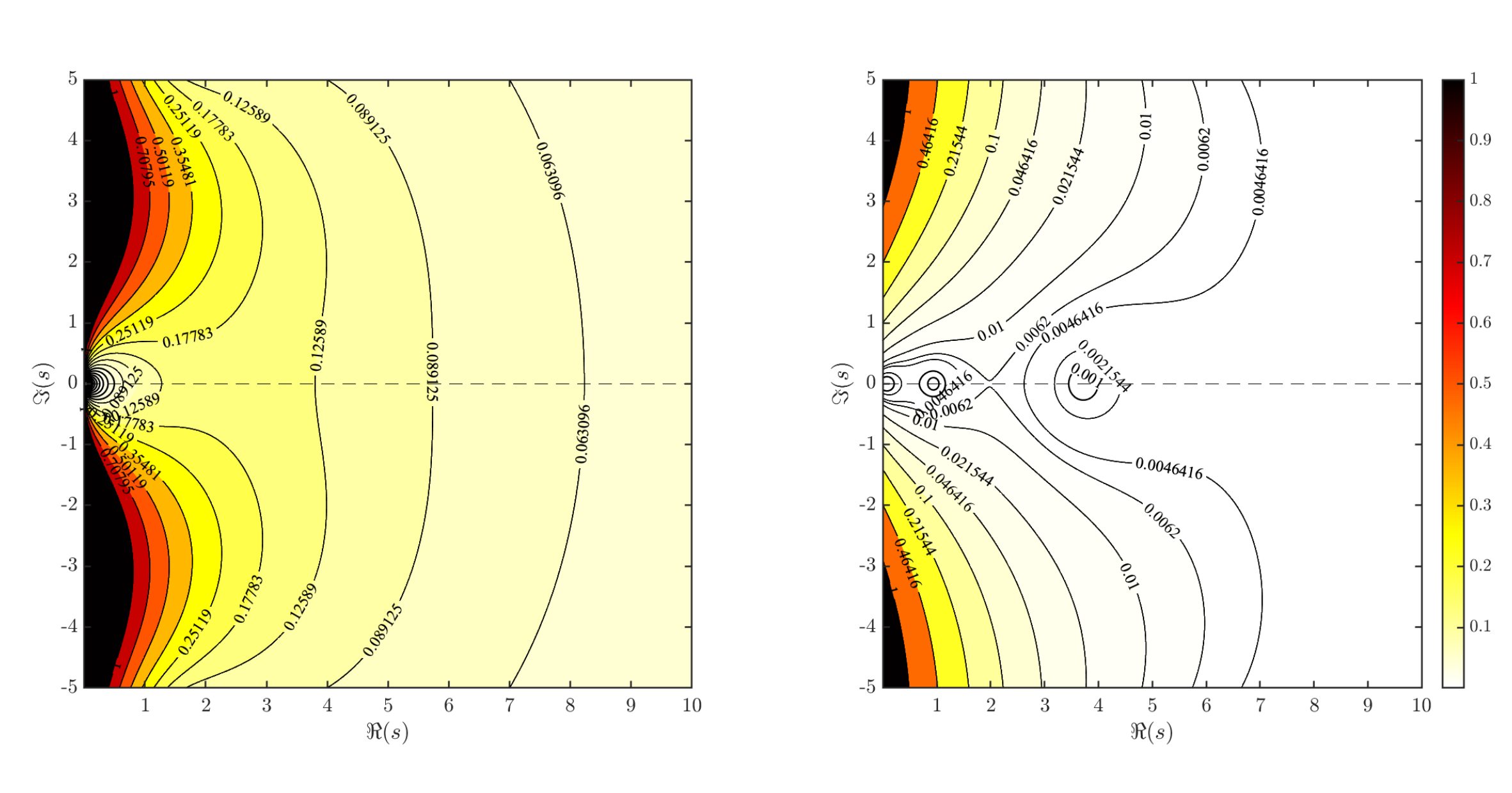}
    
    \begin{minipage}[t]{0.48\textwidth}
        \centering
        (a)
    \end{minipage}\hfill
    \begin{minipage}[t]{0.48\textwidth}
        \centering
        (b)
    \end{minipage}

    \caption{ $\kappa_e (R_1,R_2,s,10^3)$ for (a) the BDF2 and (b) the O2CP. }
    \label{fig:contour_new}
\end{figure}

\section{Numerical experiments and discussions}\label{sec:num}

\subsection{Test on the linear model}\label{sec:linear}
Consider the following initial and boundary value problem on the domain $\Omega=(0,1)$:
\begin{equation}\label{eqn:diffusion}
\partial_t u(x,t) - \partial_{xx} u(x,t) = f(x,t),\quad  (x,t) \in \Omega\times (0,T), 
\end{equation}
with $f(x,t) = \sin(\pi x)(-\pi\sin(\pi t)+\pi^2 \sin (\pi x) \cos(\pi t))$, a zero Dirichlet boundary condition and initial condition $u_0$, and the following three sets of problem data: {\rm(i)}  $T=10$, $u_0 = \chi_{(0,1/2)},$ with 
$\chi_{(0,1/2)}$ being the characteristic function of the set $(0,1/2)$; {\rm(ii)} $T=10$, $u_0(x) = \sin(\pi x)$ and $u(x,t)=\sin (\pi x) \cos(\pi t);$ {\rm(iii)} $T=1$, $u_0(x) = \sin(\pi x)$ and $u(x,t)=\sin (\pi x) \cos(\pi t).$

In the experiment, we divide the domain $\Omega$ into $1000$ equal subintervals, each of length $h=1/1000$, and employ the Galerkin FEM with linear elements. We initialize $U_{n}^{0}$ with random values uniformly drawn from $[0,1]$. Fix the three-stage Lobatto IIIC in~\eqref{eqn:RIIA3} as the FP with a time step size $\Delta t  = 0.01$, and take coarsening factors $J=20, 50$. We measure the error $e=\max_{1\leq n\leq N_{c}}\|U_{n}^{k}-U_n\|_{L^{2}\II} $ between the parareal solution $U^k_n$ and the fine solution $U_n$. Consider four CPs: BDF2 in~\eqref{eqn:BDF2}, O2CP in~\eqref{eqn:TS OCP_E}, OCP~\cite{jin2025optimizing}, and SDIRK2. The stability functions of OCP and SDIRK2 are given by
\begin{equation*} 
\begin{aligned}
    \text{(OCP): } R(s) &= \frac{{1- 0.21014s + 0.00486s^2}}{{1 + 0.78986s + 0.38283s^2}},\\
     \mbox{(SDIRK2): }  R(s) & = \dfrac{(2\gamma - 1) s + 1}{(\gamma s + 1)^{2}}, \quad \mbox{with }\gamma = \frac{2-\sqrt{2}}{2}.
\end{aligned}
\end{equation*}
The OCP and SDIRK2 are CPs in Algorithm~\ref{alg:para}, and the O2CP and BDF2 are CPs in Algorithm~\ref{alg:two-step parareal}. The numerical results are shown in Fig.~\ref{fig:Ex1}. 

In case (i), the initial data $u_0$ is nonsmooth. With $\gamma_\ell\approx0.014$ \cite{jin2025optimizing}, the parareal method with the OCP converges much faster than that with SDIRK2. Theorem~\ref{thm:error estimation} and Remark~\ref{rmk:sharper} ensure that the two-step parareal method with the O2CP and BDF2 as CPs achieves a convergence factor at least around $0.0062$ and $0.151$, respectively, see Figs.~\ref{fig:rho_BDF2} and \ref{fig:rho_OCP}. The top row of Fig.~\ref{fig:Ex1} aligns with the estimate. Note that these two CPs are implemented with one half of the coarse time step size. 

Table~\ref{tab:Ex1} compares the parareal methods using the four CPs. The convergence factor is measured for case (i) with $J=50$, cf. {Fig.~\ref{fig:Ex1} (b)}. The estimate $\kappa_e^*$ closely matches the empirical one and thus is sharp. The two-step parareal algorithm with the O2CP exhibits the smallest convergence factor and thus first achieves a maximum $L^2(\Omega)$ parareal error below $10^{-9}$, one fewer iteration than that with the OCP. Its cost is much lower than that of the OCP and only slightly higher than BDF2.  The speed-up compares the sequential FP to the parareal algorithm and is evaluated when the parareal error falls below $10^{-9}$. It is defined by 
\begin{equation}\label{eqn:speed-up}
\textbf{S}_{\text{para/seq}} = \textbf{cost}_{\text{seq}}  / (\textbf{Iter}\times(\textbf{cost}_{\text{CP}}+\textbf{cost}_{\text{FP}})),
\end{equation}
where $\textbf{cost}_{\text{CP}}$ is the average sequential cost of the CP over $[0,T]$ and $\textbf{cost}_{\text{FP}}$ is the average cost of the FP within each coarse interval. Here we have ignored the communication. The speed-up of the O2CP is the largest, indicating that Algorithm \ref{alg:two-step parareal} with the {O2CP} is effective for solving linear parabolic problems. 

In case (ii), the initial data $u_0$ is smooth. The convergence is similar to  case (i). Fig.~\ref{fig:Ex1} indicates that both single-step and two-step parareal algorithms with $J=50$ perform worse than that with $J=20$, and approach the estimate.

In case (iii), we verify the finite iteration convergence property in Theorem~\ref{thm:finite conv.}. With a fine step size $\tau = 0.01$ and the coarsening factors $J=20$ and $50$. Theorem~\ref{thm:finite conv.} predicts that the single-step and two-step parareal algorithms attain convergence at the fifth and second iterations, respectively. This is clearly observed in Fig.~\ref{fig:Ex1_finite} (zero values have been replaced with $10^{-12}$).

\begin{figure}[htbp]
\centering
\begin{minipage}[b]{0.49\textwidth}
\centering
\includegraphics[width=\textwidth,trim={0.cm 0.2cm 0.5cm 1.3cm},clip]{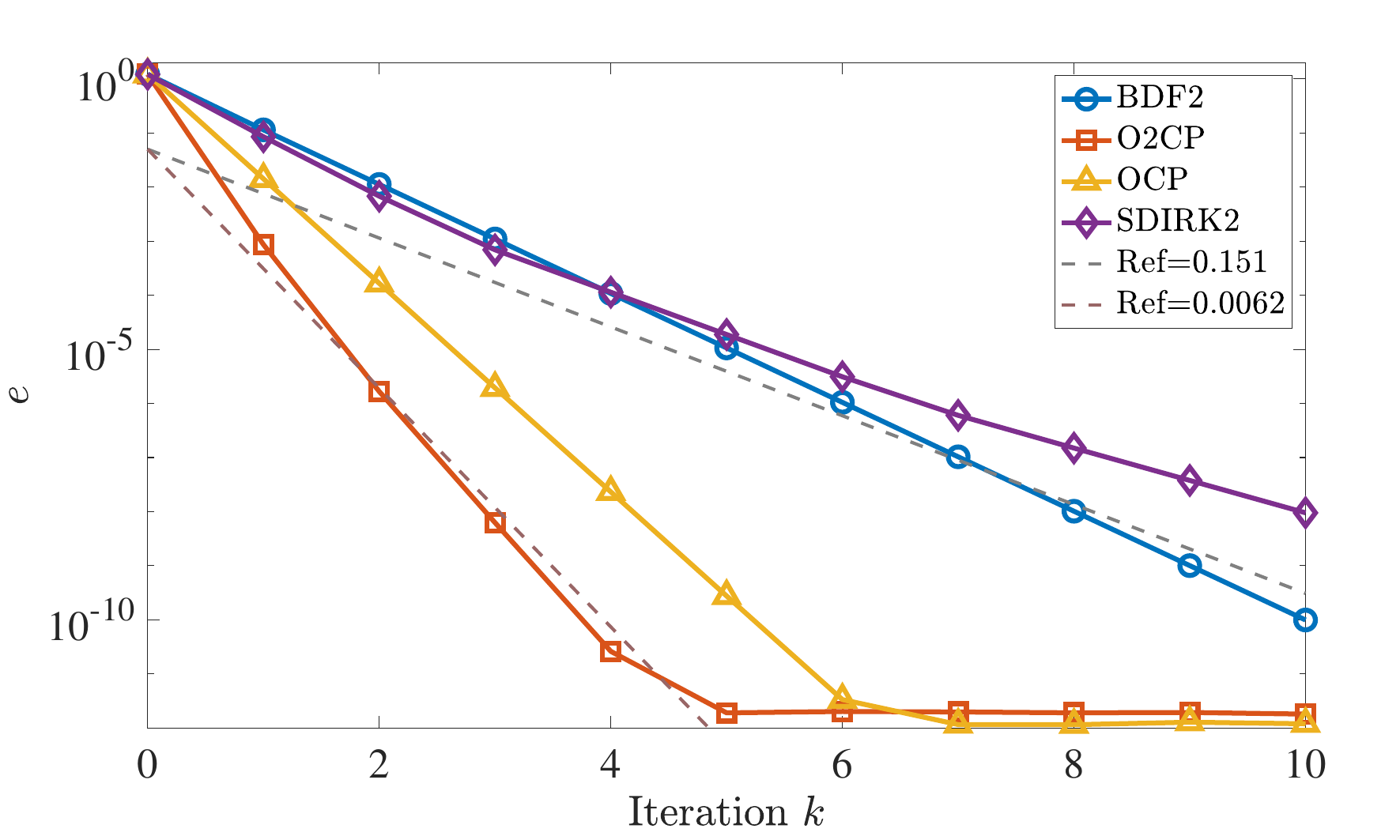}
\caption*{(a) (i), $J=20$.}
\end{minipage}
\hfill
\begin{minipage}[b]{0.49\textwidth}
\centering
\includegraphics[width=\textwidth,trim={0.cm 0.2cm 0.5cm 1.3cm},clip]{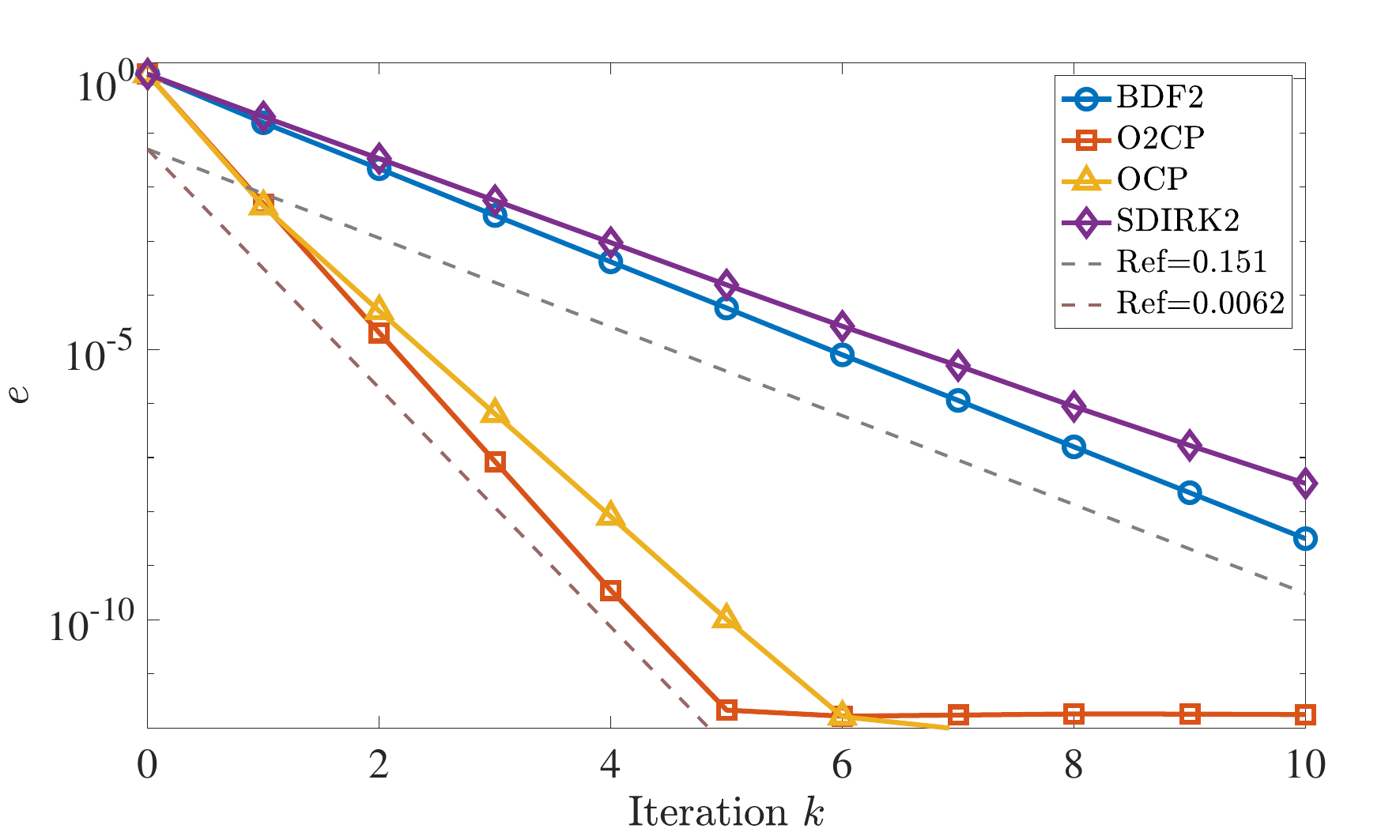}
\caption*{(b) (i), $J=50$.}
\end{minipage}
\begin{minipage}[b]{0.49\textwidth}
\centering
\includegraphics[width=\textwidth,trim={0.cm 0.2cm 0.5cm 1.3cm},clip]{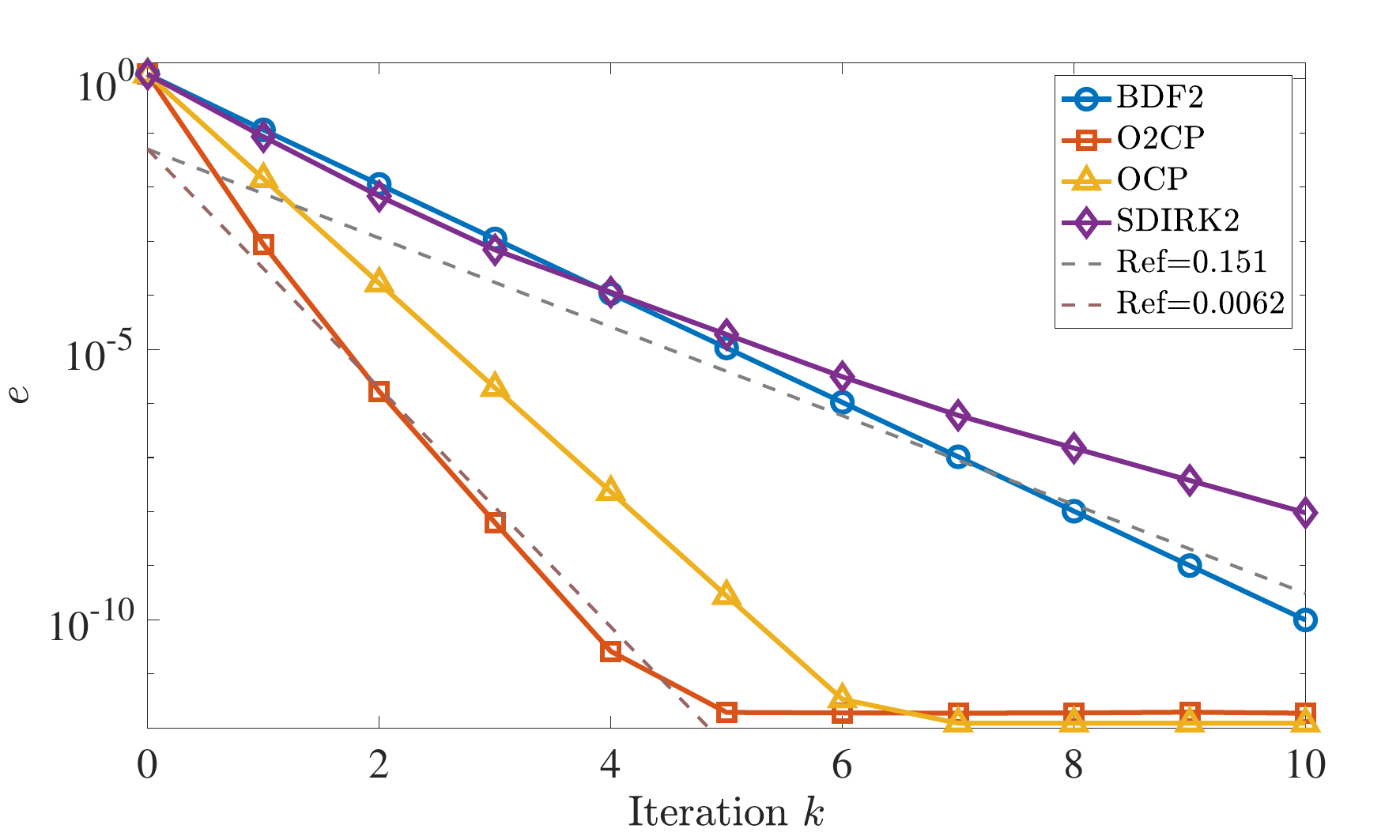}
\caption*{(c) (ii), $J=20$.}
\end{minipage}
\hfill
\begin{minipage}[b]{0.49\textwidth}
\centering
\includegraphics[width=\textwidth,trim={0.cm 0.2cm 0.5cm 1.3cm},clip]{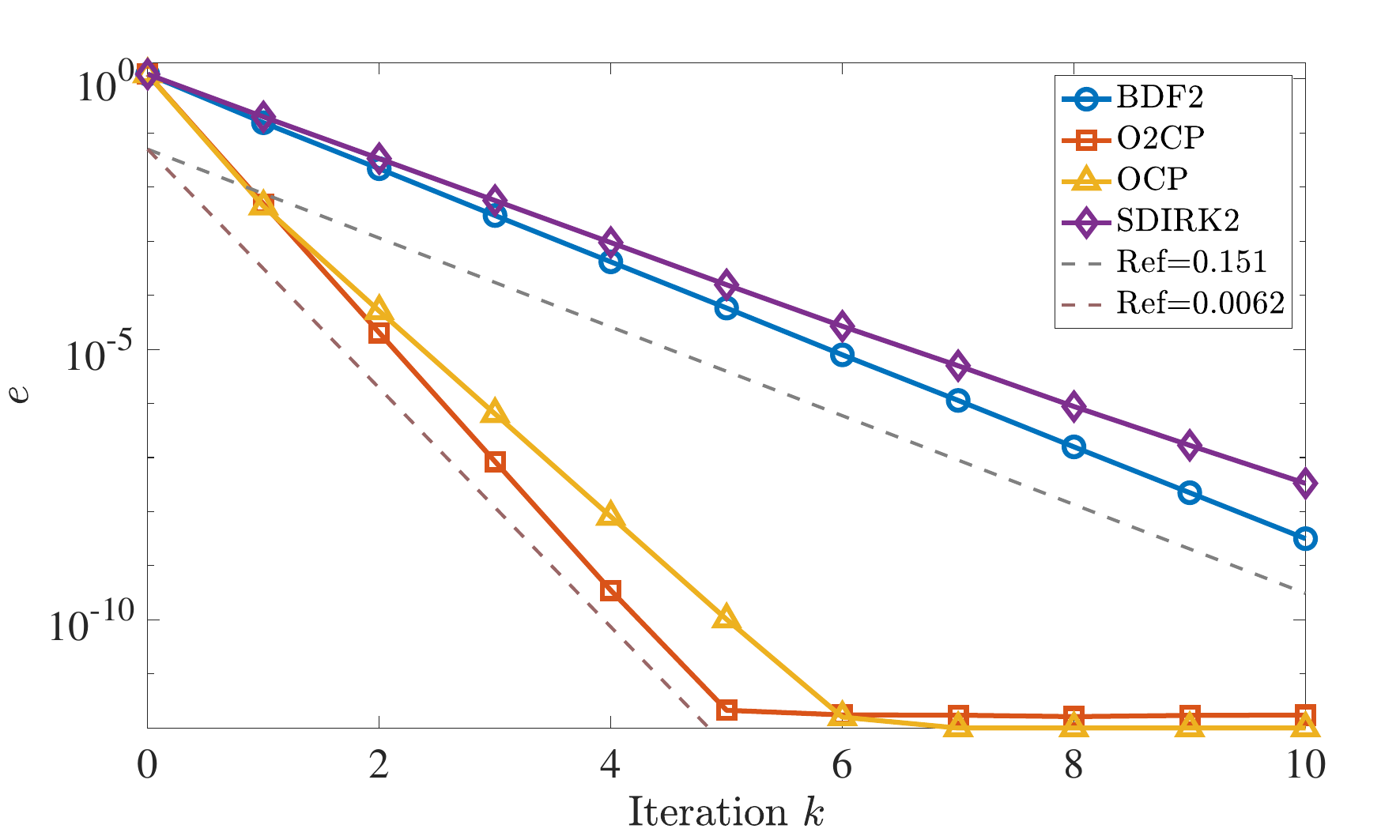}
\caption*{(d) (ii), $J=50$.}
\end{minipage}
    \caption{
    The $L^2$ error $e$ for four CPs versus the iteration $k$ for problem~\eqref{eqn:diffusion}.}
    \label{fig:Ex1}
\label{fig:Ex1_new}
\end{figure}

\begin{table}[htbp!]
\centering
\caption{Comparison of four CPs for problem \eqref{eqn:diffusion} (i): the reduced convergence factor $\gamma_e^*$ in \eqref{eqn:reduced conv.}, the factor $\kappa_e^*$ in \eqref{eqn:reduced conv.}, the empirical convergence factor~$\widehat{\gamma}_e$ with $J=50$, the cost (in seconds) of the CPs per iteration, the number of iterations required by each CP to achieve a $L^2$ error $e$ below $10^{-9}$,  and the speed-up $\mathbf{S}_{\text{para/seq}}$ defined in~\eqref{eqn:speed-up}}. 
\begin{tabular}{@{} l S[table-format=1.2] S[table-format=1.2] S[table-format=1.3] S[table-format=1.4] @{}}
\toprule
 & {SDIRK2} & {BDF2}  & {OCP} & {{O2CP}} \\
\midrule
$\gamma_e^*$ & 0.26 & 0.22 & 0.014 & \textbf{0.0064} \\
$\kappa_e^*$ & {---} & 0.15  & {---} & \textbf{0.0062} \\
$\widehat{\gamma}_e$ & 0.18 & 0.14  & 0.011 & \textbf{0.0045} \\
Cost & 0.0018 & 0.0022 & 0.0077 & \textbf{0.0034} \\
Iter. & {13} & {11}  & {5} & \textbf{4} \\
$\mathbf{S}_{\text{para/seq}}$ & {1.52} & {1.80}  & {3.84} & \textbf{4.91} \\
\bottomrule
\end{tabular}

\label{tab:Ex1}
\end{table}

\begin{figure}[htbp]
\centering
\begin{minipage}[b]{0.49\textwidth}
\centering
\includegraphics[width=\textwidth,trim={0.cm 0.2cm 0.5cm 1.3cm},clip]{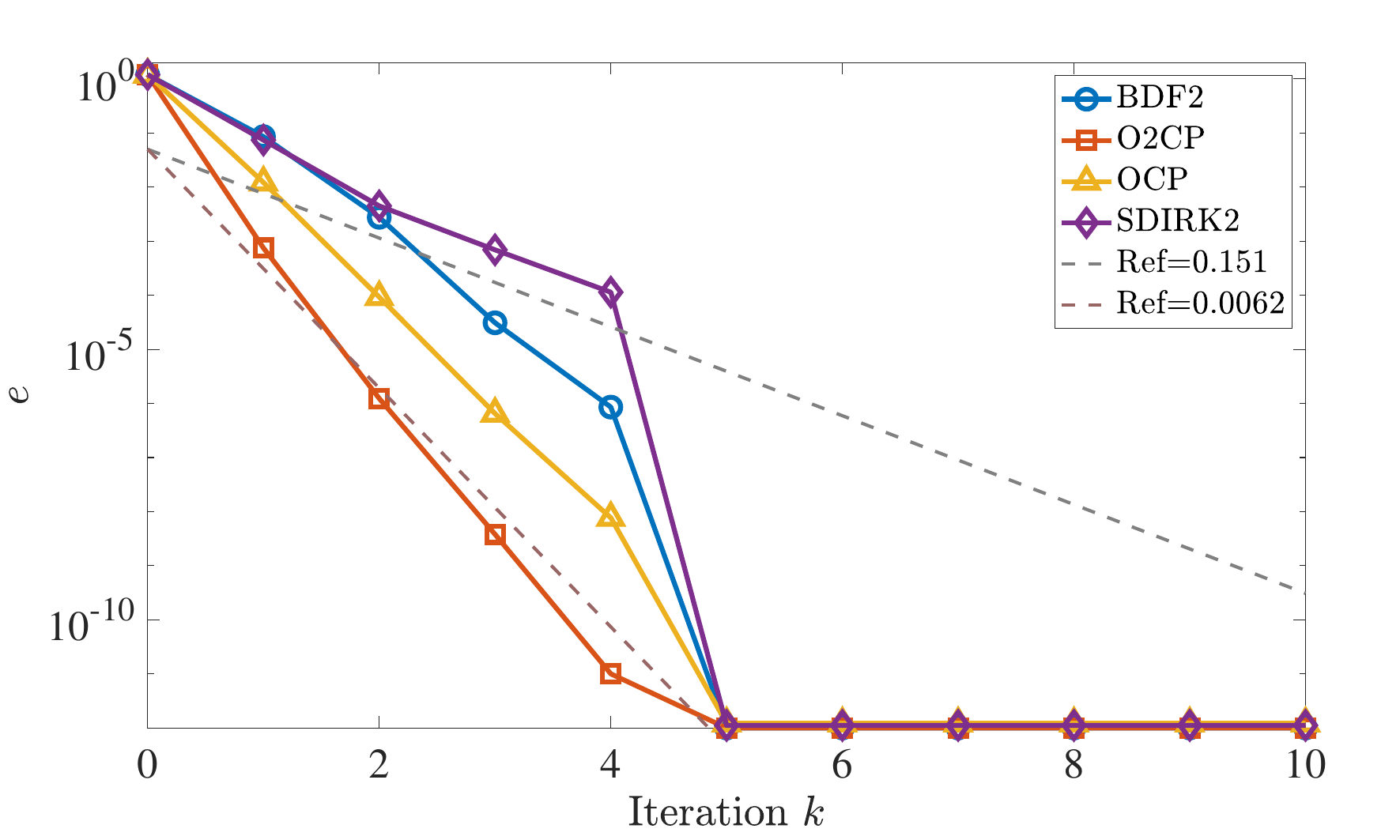}
\caption*{(a) (iii), $J=20$.}
\end{minipage}
\hfill
\begin{minipage}[b]{0.49\textwidth}
\centering
\includegraphics[width=\textwidth,trim={0.cm 0.2cm 0.5cm 1.3cm},clip]{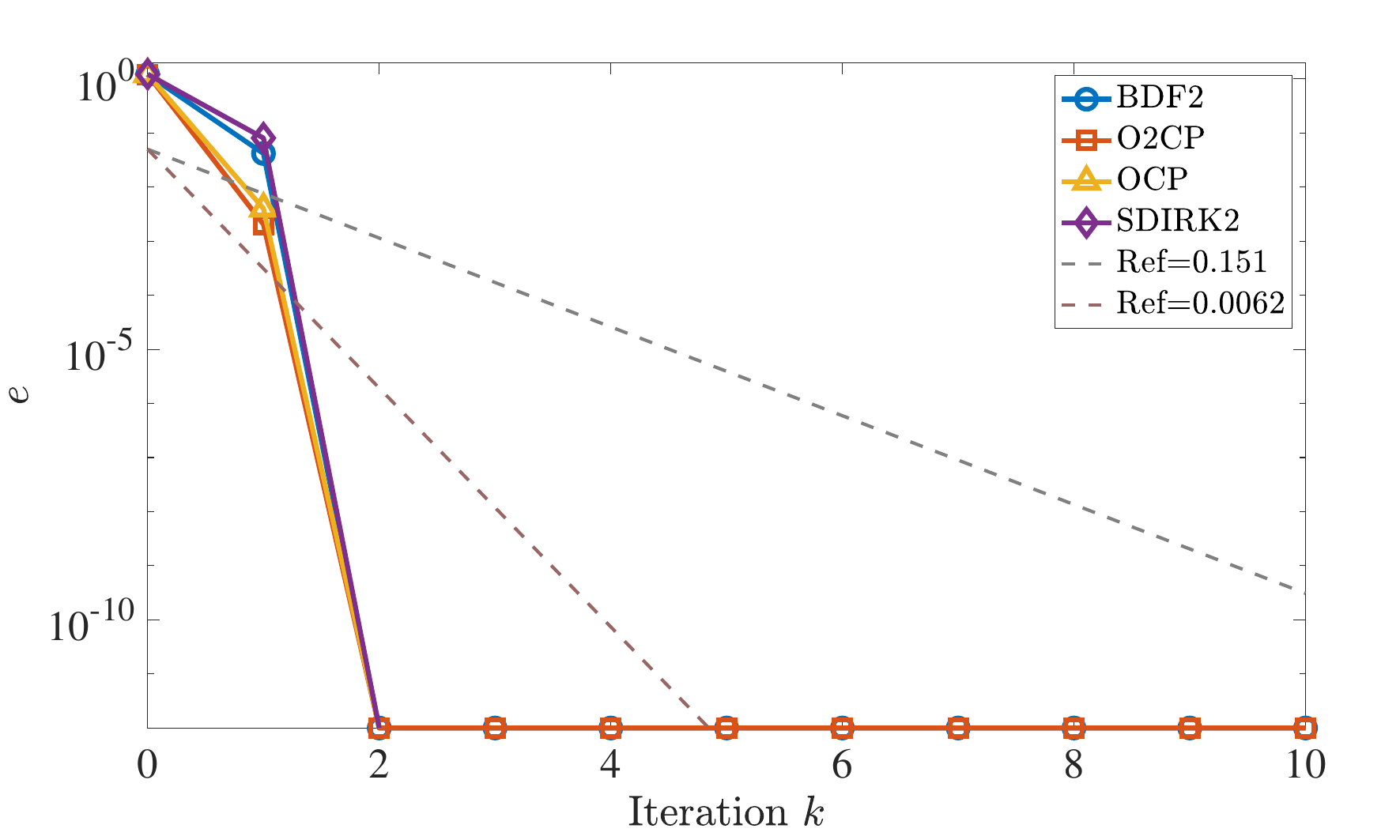}
\caption*{(b) (iii), $J=50$.}
\end{minipage}
\caption{The $L^2$ error for four CPs versus the iteration $k$ problem~\eqref{eqn:diffusion}~(c).}
    \label{fig:Ex1_finite}
\end{figure}

\subsection{Test on the nonlinear model}
Let $\Omega = (0,1)$ and $T=10$. Consider the following semilinear parabolic equation:
\begin{equation}\label{eqn:nonlinear}
\partial_{t} u=\partial_{xx} u+c_L u\left( 1-u^2 \right) + g,
\end{equation}
with a zero Dirichlet boundary condition,  
with the source $g$ chosen such that the solution $u$ is $u(x, t) = \sin(\pi x) \cos (\pi t)$. 

In the experiment, we divide the domain $\Omega$ into $1000$ equal subintervals, each of length $h=1/1000$, and employ the Galerkin FEM with linear elements. We initialize $U_{n}^{0}$ with random values uniformly drawn from $[0,1]$. Fix the three-stage Lobatto IIIC in~\eqref{eqn:RIIA3} as the FP. To study the influence of the nonlinearity parameter $c_L$ on the convergence behavior, consider $c_L \in \{1,5,10\}$, with the fine step size $\Delta t = 0.01/c_L$ and coarsening factors $J=20$ and $50$. Given two initial values $v_1$ and $v_2$ at times $T_{n}-\tau$ and $T_n$, the two-step CP computes $v_3$ for problem~\eqref{eqn:nonlinear} by
\begin{equation}\label{eqn:two-step CP,nonlinear}
  v_3 = R_1(\tau A) v_1 + R_2 (\tau A)v_2 + \tau (\alpha_2 + \beta_2 \tau A)^{-1} \sum_{i=0}^2 \beta_i f(v_{i+1},T_n + (i-1)\tau ), 
\end{equation}
with $f(u,t) = c_L u(1-u^2)+g$. The two-step solver requires solving a nonlinear Poisson equation at each time step. To reduce the sequential cost of the O2CP, we  extrapolate the nonlinearity $f$ in \eqref{eqn:two-step CP,nonlinear} by replacing the unknown $f(v_3, T_n + \tau)$ with the extrapolated value $2f(v_2, T_n) - f(v_1, T_n - \tau)$. The extrapolation-based O2CP {(O2CP-E)} computes $\bar{v}_3$ by
\begin{equation}
\begin{aligned}\label{eqn:extrapol}
    \bar{v}_3 = R_1(\tau A) v_1 + R_2(\tau A) v_2 &+ \tau (\alpha_2 +\beta_2 \tau A)^{-1} \big( (\beta_1 + 2\beta_2) f(v_2,T_n)\\
    &\quad +(\beta_0 - \beta_2) f(v_1,T_n-\tau) \big).
\end{aligned}
\end{equation}
{It requires} only one Poisson {solve}. Fig.~\ref{fig:Ex2} illustrates the performance of the parareal method with SDIRK2 and the OCP-E, alongside the two-step parareal algorithm with BDF2, {the O2CP, and O2CP-E}. Here, OCP-E denotes the coarse propagator obtained from OCP, with the nonlinear term treated by explicit extrapolation (i.e., via an explicit linearization). The parareal algorithm with SDIRK2 and the two-step one with BDF2 and the {O2CP} maintains the convergence factors in the linear cases. These three CPs are implicit. In contrast, the semi-implicit schemes (the OCP-E and {O2CP-E}) are sensitive to the nonlinearity. The two-step parareal method with the {O2CP} converges with a factor smaller than $0.0062$ across all cases, reaching machine precision at the fifth iteration, much faster than the others. The two-step parareal method with the {O2CP} is affected by nonlinearity but still faster than standard solvers and incurs lower sequential costs, i.e., the {O2CP-E is efficient} for weakly nonlinear problems to maximize efficiency. While the parareal method with the OCP-E shows a small convergence factor in linear scenarios, it performs poorly under strong nonlinearity and even underperforms SDIRK2 in initial iterations, cf. {Fig.~\ref{fig:Ex2} (vi)}. 

We further analyze the performance of the five CPs in Table~\ref{tab:compare}. For all tested cases with $c_L = \{1,5,10 \}$ and $J=20,50$, the empirical convergence factor $\widehat{\gamma}_e$ of the O2CP remains smaller than its theoretical value. Although $\widehat{\gamma}_e$ of the O2CP-E exceeds the theoretical value, the method still outperforms the classical choices. The costs of SDIRK2, BDF2 and O2CP are comparable, since all three require implicit iterations and thus are more expensive than the semi-explicit schemes (i.e., the OCP-E and O2CP-E). The O2CP-E runs faster than the OCP-E since each OCP-E step involves solving two Poisson equations with complex coefficients. The two-step parareal method with the O2CP exhibits excellent stability with the nonlinearity strength and $J$, and requires only $4$ iterations to achieve $e<10^{-9}$. Moreover, the O2CP-E requires fewer iterations than the classical schemes across all case and is at least two iteration ahead. 

\begin{figure}[htbp!]
\centering
\begin{minipage}[b]{0.49\textwidth}
\centering
\includegraphics[width=\textwidth,trim={0.cm 0.2cm 0.5cm 1.3cm},clip]{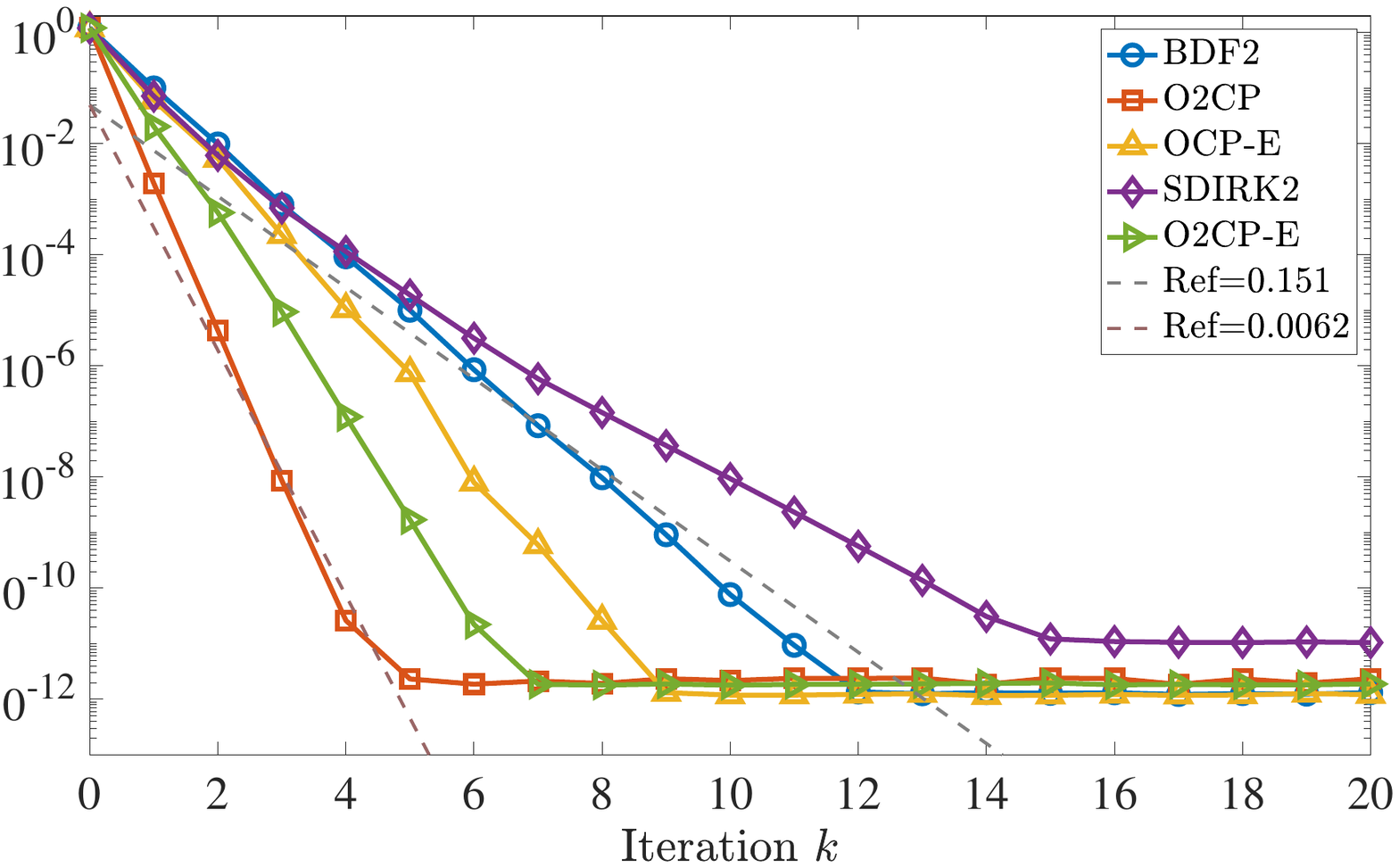}
\caption*{(i) $c_{L}=1$, $J=20$, $\Delta t = 0.01$.}
\end{minipage}
\hfill
\begin{minipage}[b]{0.49\textwidth}
\centering
\includegraphics[width=\textwidth,trim={0.cm 0.2cm 0.5cm 1.3cm},clip]{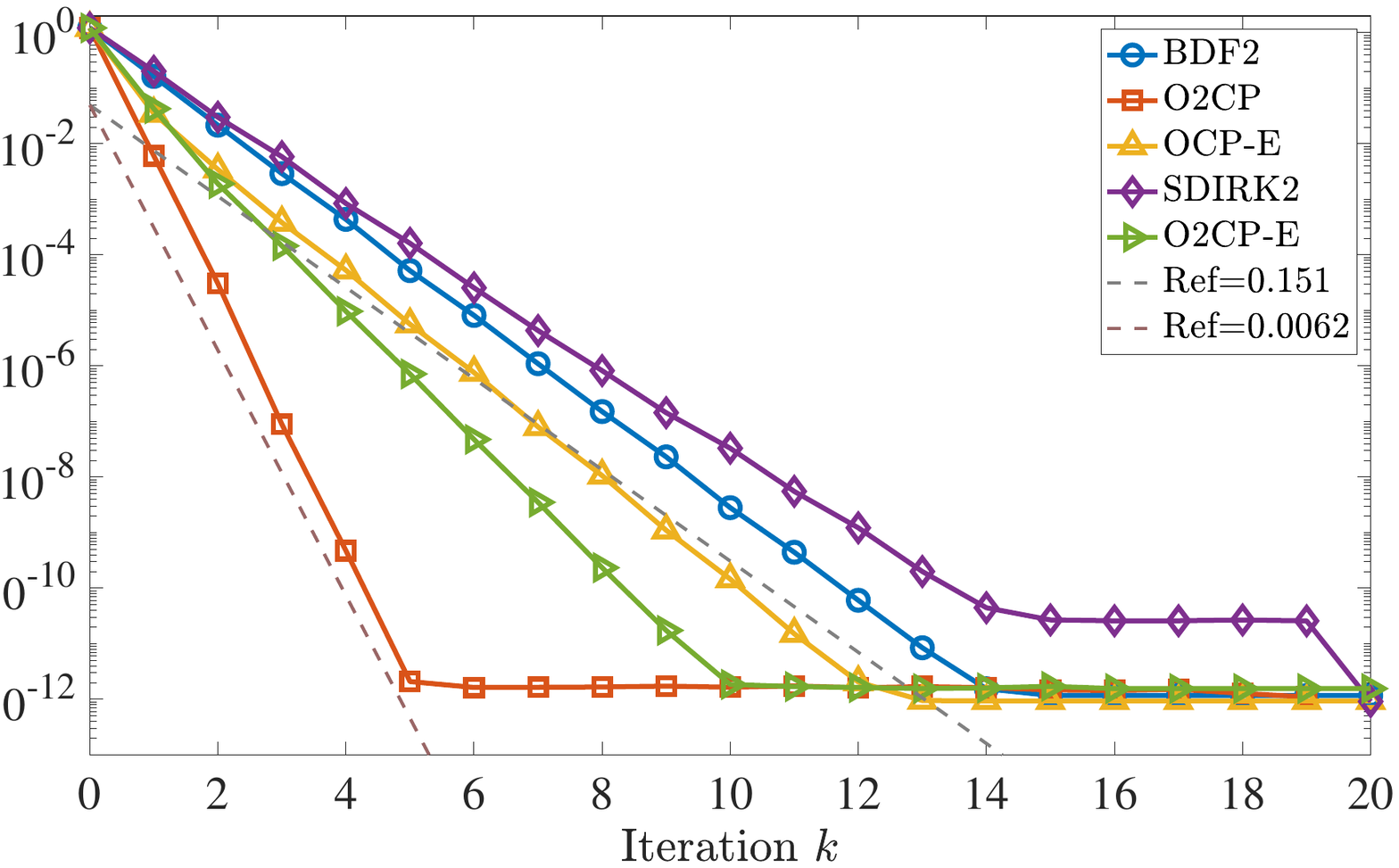}
\caption*{(ii) $c_{L}=1$, $J=50$, $\Delta t = 0.01$.}
\end{minipage}

\vspace{1em} 

\begin{minipage}[b]{0.49\textwidth}
\centering
\includegraphics[width=\textwidth,trim={0.cm 0.2cm 0.5cm 1.3cm},clip]{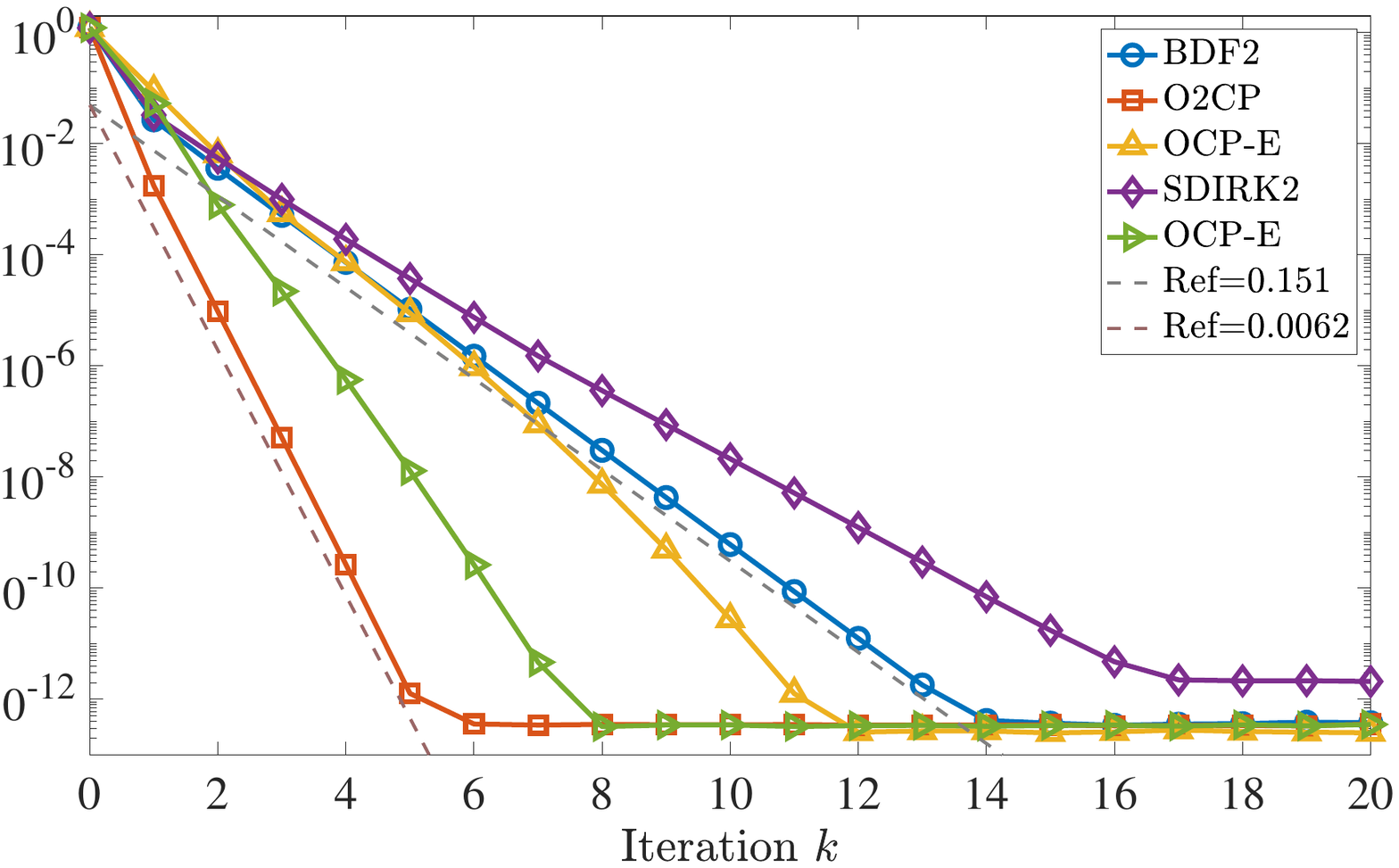}
\caption*{(iii) $c_{L}=5$, $J=20$, $\Delta t = 0.002$.}
\end{minipage}
\hfill
\begin{minipage}[b]{0.49\textwidth}
\centering
\includegraphics[width=\textwidth,trim={0.cm 0.2cm 0.5cm 1.3cm},clip]{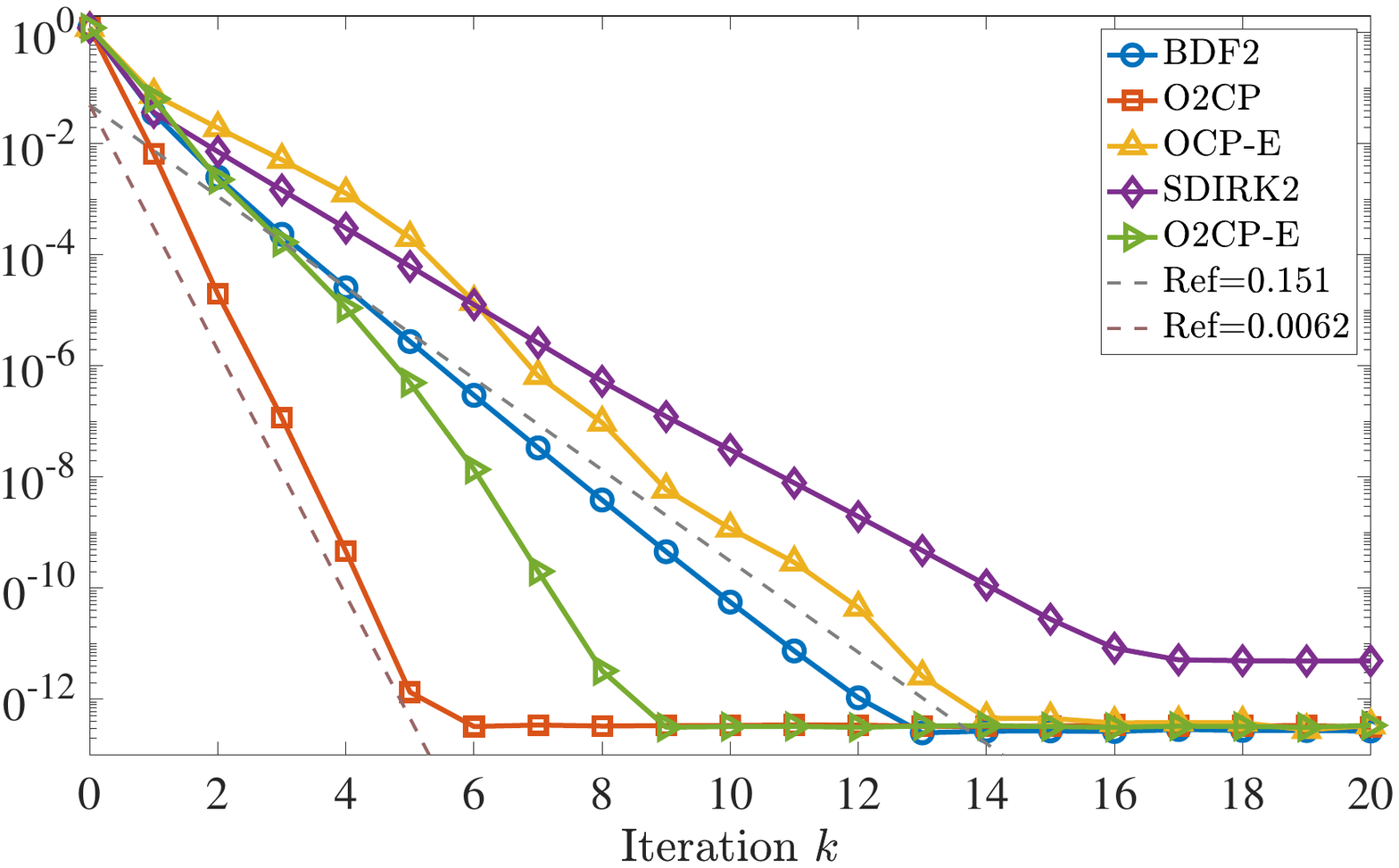}
\caption*{(iv) $c_{L}=5$, $J=50$, $\Delta t = 0.002$.}
\end{minipage}

\vspace{1em} 

\begin{minipage}[b]{0.49\textwidth}
\centering
\includegraphics[width=\textwidth,trim={0.cm 0.2cm 0.5cm 1.3cm},clip]{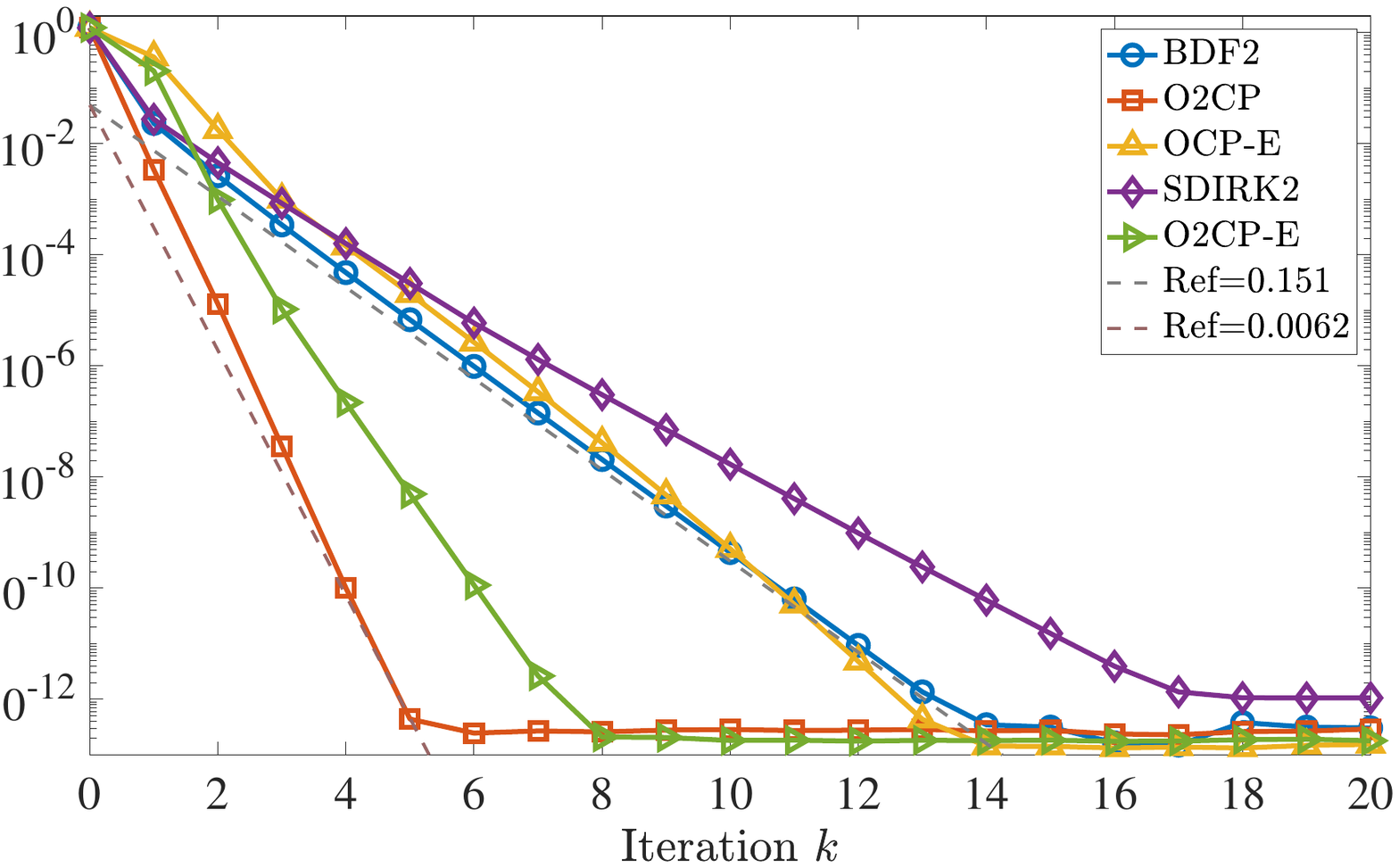}
\caption*{(v) $c_{L}=10$, $J=20$, $\Delta t = 0.001$.}
\end{minipage}
\hfill
\begin{minipage}[b]{0.49\textwidth}
\centering
\includegraphics[width=\textwidth,trim={0.cm 0.2cm 0.5cm 1.3cm},clip]{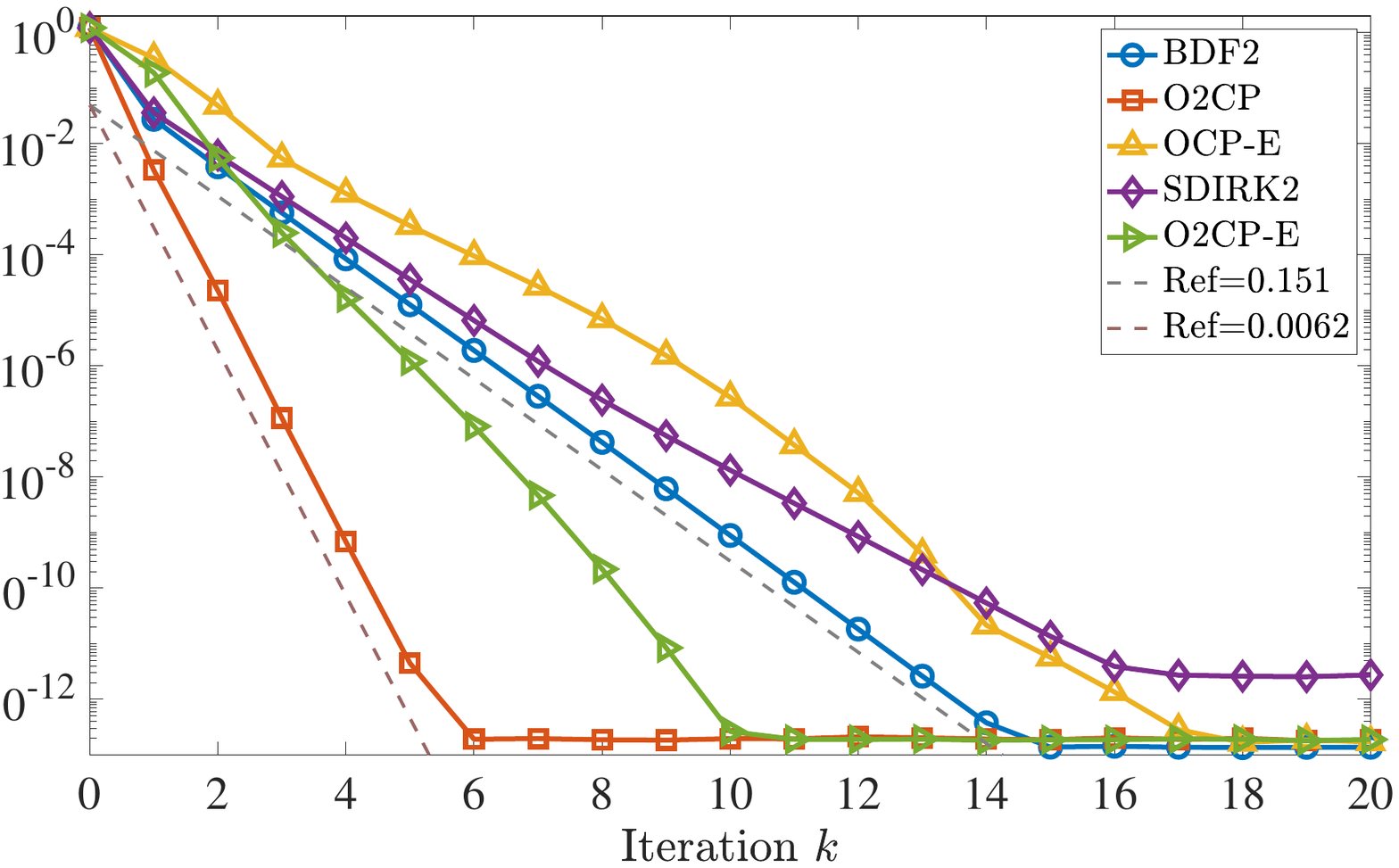}
\caption*{(vi) $c_{L}=10$, $J=50$, $\Delta t = 0.001$.}
\end{minipage}
\caption{The $L^2$ error $e$ for five CPs versus the iteration $k$ for problem~\eqref{eqn:nonlinear} with $c_L \in \{ 1,5,10\}$ and $J \in \{ 20,50\}$.}\label{fig:Ex2}
\end{figure}

\begin{table}[htbp!]
\centering
\caption{The comparison of five CPs for problem~\eqref{eqn:nonlinear} with $c_L \in \{1,5,10\}$, using $J=20,50$: empirical convergence factor $\widehat{\gamma}_e$, cost (in seconds) of the CPs per iteration, and the number of iterations required to achieve an $L^2$ error $e$ below $10^{-9}$.}
\label{tab:compare}
\begin{tabular}{c | c c c c c c}
\toprule
{CP} & \multicolumn{2}{c}{$c_L=1$} & \multicolumn{2}{c}{$c_L=5$} & \multicolumn{2}{c}{$c_L=10$} \\
\midrule
& \multicolumn{6}{l}{$\widehat{\gamma}_e$} \\[2pt]
$J$ & 20 & 50 & 20 & 50 & 20 & 50 \\[3pt]
SDIRK2 & 0.19 & 0.18 & 0.20 & 0.21 & 0.21 & 0.20 \\
BDF2 & 0.097 & 0.14 & 0.13 & 0.099 & 0.13 & 0.13 \\
OCP-E & 0.054 & 0.11 & 0.096 & 0.16 & 0.13 & 0.23 \\
O2CP & \textbf{0.0022} & \textbf{0.0046} & \textbf{0.0044} & \textbf{0.0046} & \textbf{0.0030} & \textbf{0.0051} \\
O2CP-E & \textbf{0.018} & \textbf{0.064} & \textbf{0.028} & \textbf{0.057} & \textbf{0.041} & \textbf{0.064} \\
\midrule
& \multicolumn{6}{l}{Cost} \\[2pt]
$J$ & 20 & 50 & 20 & 50 & 20 & 50 \\[3pt]
SDIRK2 & 0.18 & 0.076 & 0.76 & 0.35 & 1.51 & 0.64 \\
BDF2 & 0.21 & 0.092 & 0.95 & 0.42 & 1.85 & 0.78 \\
OCP-E & 0.021 & 0.0082 & 0.11 & 0.043 & 0.22 & 0.085 \\
O2CP & \textbf{0.24} & \textbf{0.099} & \textbf{1.08} & \textbf{0.46} & \textbf{2.05} & \textbf{0.87} \\
O2CP-E & \textbf{0.016} & \textbf{0.0062} & \textbf{0.077} & \textbf{0.032} & \textbf{0.17} & \textbf{0.061} \\
\midrule
& \multicolumn{6}{l}{Iter.} \\[2pt]
$J$ & 20 & 50 & 20 & 50 & 20 & 50 \\[3pt]
SDIRK2 & 12 & 13 & 13 & 13 & 12 & 12 \\
BDF2 & 9 & 11 & 10 & 9 & 10 & 10 \\
OCP-E & 7 & 10 & 9 & 11 & 10 & 13 \\
O2CP & \textbf{4} & \textbf{4} & \textbf{4} & \textbf{4} & \textbf{4} & \textbf{4} \\
O2CP-E & \textbf{6} & \textbf{8} & \textbf{6} & \textbf{7} & \textbf{6} & \textbf{8} \\
\bottomrule
\end{tabular}
\end{table}

\section{Conclusion}
In this work, we have developed a two-step parareal framework and an optimization-based acceleration strategy for parabolic problems. By extending the classical parareal algorithm to the two-step setting and constructing optimized two-step coarse propagators (O2CPs) that satisfy only the consistency condition, we have achieved a convergence factor of $0.0064$, at the modest cost of two Poisson solves per coarse time step. The analysis confirms the predicted rate. Numerical experiments on both linear and semilinear parabolic equations show the robustness and superior performance of the {O2CPs}. 

The present work opens several promising directions for future research, including the extension to multi-step propagators, fully nonlinear error analysis, and the incorporation of adaptive time-stepping within the multi-step parareal framework. 


\section*{Acknowledgements}
The work of G. Li is supported by Hong Kong Research Grants Council (Project 17317022). The work of K. Zhang is supported by the NSF of China under the grant No. 12271207, and by the fundamental research funds for the central universities. The work of Z. Zhou is supported by by National Natural Science Foundation of China (Project 12422117),
Hong Kong Research Grants Council (15302323) and an internal grant of Hong Kong Polytechnic University (Project ID: P0053938, Work Programme: 4-ZZVA).

\bibliographystyle{plain}
\bibliography{reference}

\end{document}